\documentclass{article}
\usepackage[utf8]{inputenc}
\usepackage{amsfonts, mathtools, nccmath, amsmath, amssymb, amsthm, graphicx, multicol, commath, bbm, mathrsfs, bm}
\usepackage{authblk}
\usepackage{setspace}
\usepackage{braket}
\usepackage{esint}
\usepackage{bibentry}
\usepackage{color,dsfont}
\usepackage[top=1in, bottom=1in, left=1in, right=1in]{geometry}
\usepackage{hyperref}
\usepackage{tikz, graphicx, subfig, booktabs}
\usepackage{stackengine}
\usepackage{pgfplots}

\pgfplotsset{width=10cm,compat=1.9}
\pgfplotsset{compat=newest}

\tikzset{%
    pics/myrec/.style n args={2}{code={%
            \draw (0,0) rectangle (#1,#2);
    }},
}
\AtBeginDocument{\colorlet{normalcolor}{.}\def\normalcolor{\color{normalcolor}}} 

\newif\ifdraft
\draftfalse 
\ifdraft
\usepackage[showzone=true,showisoZ=false]{datetime2}
\DTMsetdatestyle{mmddyyyy}
\DTMsetup{datesep={/}}
\date{DRAFT: \today}
\else
\date{September 12, 2024} 
\usepackage{fancyhdr}
\fi

\hypersetup{colorlinks,breaklinks,
             linkcolor=blue,urlcolor=blue,
             anchorcolor=blue,citecolor=blue}

\usepackage{cite}
\title{Convergence and non-convergence in a nonlocal gradient flow}
\author{Sangmin Park\footnote{Email address: sangminp@andrew.cmu.edu} }
\author{Robert L. Pego\footnote{Email address: rpego@cmu.edu} }

\affil{%
Department of Mathematical Sciences\\ 
Carnegie Mellon University\\
Pittsburgh, PA 15213.}

\newcommand{\half}{\frac12}
\newcommand{\R}{\mathbb{R}}

\newcommand{\N}{\mathbb{N}}

\newcommand{\phib}{\boldsymbol{\phi}}
\newcommand{\uu}{\boldsymbol{u}}
\newcommand{\vv}{\boldsymbol{v}}

\newcommand{\uueq}{\uu^{\rm eq}}
\newcommand{\ueq}{u^{\rm eq}}
\newcommand{\barfeq}{\bar f^{\rm eq} }

\newcommand{\E}{\mathcal{E}}

\newcommand{\Fcal}{\mathcal{F}}

\newcommand{\one}{\mathds{1}}
\newcommand{\eps}{\varepsilon}

\newcommand{\M}{\mathcal{M}}
\newcommand{\calN}{\mathcal{N}}
\newcommand{\calQ}{\mathcal{Q}}

\newcommand{\ip}[1]{\langle #1 \rangle}
\newcommand{\D}{\partial}
\newcommand{\inv}{^{-1}}

\newcommand{\jnex}{j_m}

\definecolor{byzantine}{rgb}{0.74, 0.2, 0.64}

\newcommand{\blue}{\normalcolor}

\newcommand{\purp}{\normalcolor}
\newcommand{\nc}{\normalcolor}

\newcommand{\Loja}{{\L{}ojasiewicz}}
\newcommand{\Sengul}{{\c Seng\"ul}}

\DeclareMathOperator*{\osc}{osc}

\newtheorem{theorem}{Theorem}
\newtheorem{lemma}{Lemma}
\newtheorem{proposition}{Proposition}
\newtheorem{corollary}{Corollary}

\theoremstyle{definition}
\newtheorem{remarkx}{Remark}
\newenvironment{remark}{\pushQED{\qed}\remarkx}{\popQED\endremarkx}

\numberwithin{equation}{section}

\begin{document}

\maketitle
 \begin{center}
 {\em Dedicated to Sir John Ball}
 \end{center} \smallskip

\begin{abstract}
   We study the asymptotic convergence \blue as $t\rightarrow\infty$ \nc of solutions of $\partial_t u=-f(u)+\int f(u)$,
   a nonlocal differential equation that is formally a gradient flow in a constant-mass subspace of $L^2$ arising from simplified models of phase transitions. 
   In case the solution takes finitely many values, we provide a new proof of stabilization that uses a \Loja-type gradient inequality near a degenerate curve of equilibria.
   Solutions with infinitely many values in general {\em need not} converge to equilibrium, however, 
   which we demonstrate by providing counterexamples for piecewise linear and cubic functions $f$. 
   Curiously, the exponential {\em rate} of convergence in the finite-value case can jump from order $O(1)$ to arbitrarily small values upon perturbation of parameters.
 \end{abstract}


 \newcommand\barbelow[1]{\stackunder[1.2pt]{$#1$}{\rule{.8ex}{.075ex}}}

\section{Introduction}
Let $(\Omega,\Fcal,\nu)$ be a probability space, and assume
$f\colon\R\to\R$ is locally Lipschitz and piecewise $C^1$.  
This paper investigates the asymptotic behavior as $t\to\infty$ of 
bounded solutions to the nonlocal differential equation
 \begin{equation}\label{eq:ODE_nonlocal}
     \partial_t u(x,t) = -f(u(x,t))+ 
     \int_{\Omega}f(u(y,t))\,d\nu(y)\,,
     \quad x\in\Omega,\ t\geq 0.
 \end{equation}
While our main concern involves general nonlinear functions $f$,
key examples to bear in mind are nonmonotonic polynomials and piecewise-linear 
functions.

The initial-value problem for \eqref{eq:ODE_nonlocal} is well-posed locally in time in $B(\Omega)$, 
the Banach space of bounded measurable functions on $\Omega$ equipped with the supremum norm.
The solution exists globally in time and remains uniformly bounded if, for example, 
the initial data $u(x,0)=u_0(x)$ lies in an interval $[a,b]$ with the property that  
$f(a)\le f(s)\le f(b)$ {for all $s\in[a,b]$}, 
as will be seen below.

The nonlocal term in \eqref{eq:ODE_nonlocal} ensures that the mean of the solution is conserved, as
\begin{equation}\label{e:dtmass}
 \frac{d}{dt}\int_\Omega \blue u(x,t)\nc\,d\nu(x) 
 =\int_\Omega \left(-f(\blue u(x,t)\nc)+ \bar f(t) \right)d\nu(x) =0,
\end{equation}
where
 \begin{equation} \label{d:fbar}
     \bar f(t)=\int_\Omega f(u(x,t))\,d\nu(x).
 \end{equation}
We can view \eqref{eq:ODE_nonlocal} formally as the equation of 
 $L^2$-gradient flow constrained by fixing the mean:
 Let $F$ denote the antiderivative of $f$ --- i.e.,
 \begin{equation}\label{def:W}
     F(x)=\int_0^x f(y)\,dy,
 \end{equation}
 and define the energy $\E$ by
 \begin{equation}\label{def:E}
     \E(u)=\int_\Omega F(u(x))\,d\nu(x)=\int_\Omega \int_0^{u(x)}f(y)\,dy\,d\nu(x).
 \end{equation}
 Then the equation \eqref{eq:ODE_nonlocal} can be written formally in the form
 \begin{equation}
     \D_t u = -\calQ\nabla\E(u),
 \end{equation}
 where $\nabla\E(u)=f\circ u$ is formally 
the $L^2$-gradient of $\E$ at $u$, and $\calQ$ is the $L^2$-orthogonal projection
on the space of functions with mean zero.

Due to this constrained gradient structure, the energy 
is dissipated along solutions of \eqref{eq:ODE_nonlocal},
with 
\begin{align*}
    \blue\frac{d}{dt}\E(u) =
    \int_\Omega f(u)\D_t u\,d\nu = 
    \int_\Omega (f(u)-\bar f(t))\D_t u\,d\nu 
    = -\int_\Omega |\D_t u|^2\,d\nu\,.\nc
\end{align*}
Hence for a bounded solution the limit $\E_\infty = \lim_{t\to\infty}\E(u(t))$ exists, and we have
\begin{equation}
   \E_\infty + \int_0^\infty \blue\int_\Omega |\D_t u|^2 \,d\nu\,dt \nc= \E(u_0)\,.
\end{equation}
By \eqref{eq:ODE_nonlocal}, $\D_t u$ is uniformly Lipschitz in $t$, so  \blue$\int_\Omega|\D_t u|^2\,d\nu$\nc is as well, 
whence it follows that 
\begin{equation}\label{limdt2}
    \int_\Omega \blue|\D_t u|^2\nc\,d\nu \to 0 \quad\text{as $t\to\infty$.}
\end{equation}
Then it follows any limit point of the orbit $\{u(\cdot,t)\}_{t\ge0}$ 
(in the $L^2$ sense) must be an equilibrium, a (possibly discontinuous) function $\hat u$ such that 
$f(\hat u(x))$ is a.e.~a constant.

The main question that we resolve herein is this: 
\begin{equation}\label{mainQ}
\text{Does $u(\cdot,t)$ necessarily converge to a {\em single} equilibrium as $t\to\infty$?}
\end{equation}
It is well-known that solutions of gradient systems need not converge
in general, even in $\R^2$ \cite[p.~13]{PalisdeMelo82}.
But in the paper \cite{Pego92}, the second author proved that  for solutions of \eqref{eq:ODE_nonlocal}
the answer is {\em yes, assuming} the initial data $u_0$ has finite range, 
taking only finitely many values $u_1^0,\ldots, u_N^0$.  
In that case \eqref{eq:ODE_nonlocal} is equivalent to a finite-dimensional system
for $\uu(t)=(u_1(t),\ldots,u_N(t))$ in $\R^N$.
In \cite{Pego92},  the solution's $\omega$-limit set is shown to contain points 
in a normally hyperbolic curve of equilibria,
and a theorem of Hale and Massat \cite{HaleMassatt81} is invoked to conclude convergence
as $t\to\infty.$

As pointed out by \Sengul~\cite{Sengul2021nonlinear},  the theorem of  
Hale and Massat used in \cite{Pego92} was improved by Hale and Raugel \cite{HaleRaugel92}, 
and this could also improve the convergence proof in \cite{Pego92} in the finite range case.
One thing we provide in the present paper is a different and considerably simpler proof of convergence in the finite range case,
based on a {\em gradient inequality} of the form
\begin{equation}\label{i:grad1}
 c \|\E(u)-\E(\hat u)\|^{1/2} \le  \|\calQ \nabla \E(u)\| \,,
\end{equation}
which is proved valid for $u$ on the orbit near a ``regular'' equilibrium $\hat u$ 
in the $\omega$-limit set, \purp which is guaranteed to exist \nc under the assumption that $\bar f(t)$ fails to converge.
The use of gradient inequalities to analyze convergence of gradient flows was pioneered by 
\Loja~\cite{lojasiewicz1965ensembles} and Simon~\cite{Simon83},
and has since expanded greatly in the the field of optimization~\cite{AttouchEA2013}
and in the analysis of dynamics in PDE~\cite{HarauxJendoubi2015}.
The proof of such inequalities in general involves a deep study of objects 
such as subanalytic sets and o-minimal structures \cite{BolteEA2007,DenkowskaDenkowski2015}.
But in our case, a proof based on simple Taylor approximation works, since 
we use \eqref{i:grad1} not for arbitrarily degenerate equilibria $\hat u$, but 
only for \purp curves \nc of equilibria that, although they are not isolated, 
correspond to regular values of $f$. 
\purp
This is similar to proofs of gradient estimates near nondegenerate manifolds of equilibria 
by Simon~\cite[Lemma 1, p.~80]{Simon96} and Haraux and Jendoubi~\cite[Thm. 2.1]{HarauxJendoubi2007}.
Such arguments were generalized by Chill to reduce verification of gradient estimates
to a ``critical manifold,'' see~\cite[Thm.~3.10]{Chill2003}.
\nc

Our main result, however, is that the general answer to the main question~\eqref{mainQ} 
is {\em no!} --- It is possible that $u$ fails to converge if $u_0$ takes infinitely many values.  
We construct counterexamples to convergence in cases when 
$f$ is piecewise-linear or a cubic polynomial, having an ``$N$-shaped'' graph.  
Our constructions are motivated by the observation
that perturbations (arbitrarily small in $L^2$) of certain degenerate unstable equilibria 
can cause the value of $\bar f(t)$ to eventually drift a finite distance either up or down.
An infinite number of such perturbations can then be superimposed to cause 
$\bar f(t)$ to oscillate, slower and slower, with no limit.

\subsection{Related works}

Equation~\eqref{eq:ODE_nonlocal} is a simplified model for 
dissipative dynamics in a number of models of phase transitions that are related to each other. 
These include models of viscoelastic materials~\cite{AndBall82,Pego87,BallSengul15}, 
models of formation of material microstructure~\cite{BallEtal91,FrieseckeMcLeod1996,FrieseckeMcLeod1997},
regularized forward-backward diffusion models~\cite{Novick-CohenPego91},
and shear flows in non-Newtonian fluids~\cite{NohelEtal90,NohelPego93}.
\Sengul\ has recently reviewed work on nonlinear viscoelastic models of strain rate type~\cite{Sengul2021nonlinear}.

In order to ensure convergence of solutions in a problem of viscoelasticity,
Andrews and Ball~\cite{AndBall82} introduced a hypothesis that they called a 
{\em nondegeneracy condition}, which works also for solutions of \eqref{eq:ODE_nonlocal}.
To explain,  suppose for simplicity that $f$ is piecewise monotone, so that 
for $z$ in any bounded set of $\R$, the equation $f(z)=s$ has a
finite number $M=M(s)$ of roots $z_1(s)<z_2(s)<\ldots<z_M(s)$,
where $M$ is piecewise continuous jumping a finite number of times.
Then the {nondegeneracy condition} requires that no nonzero linear combination
of $z_1,\ldots,z_M$ is constant on any common interval of definition.
For counterexamples to convergence as constructed in this paper, 
it is important that the nondegeneracy condition be violated. 
This is indeed the case however if, e.g., $f$ is any piecewise linear function, 
or a \blue nonmonotonic \nc cubic polynomial (since then the sum of the roots $z_1+z_2+z_3$ is constant).

In 2015, Ball and \Sengul\ published an in-depth study~\cite{BallSengul15} of an equation of the form exactly as in \eqref{eq:ODE_nonlocal}  in the context of quasistatic nonlinear viscoelasticity in one space dimension. 
In this context, the variable $u$ represents the material strain and should remain positive.
For the measure space $\Omega=[0,1]$ with Lebesgue measure (or any Borel-isomorphic space),
they establish that \eqref{eq:ODE_nonlocal} is well-posed in the positive cone of $L^2(\Omega)$ 
when $F$ is $\lambda$-convex (i.e., $F(u)+\frac12\lambda u^2$ is convex) and $f(u)\to-\infty$ as $u\downarrow0$,
by making use of a one-sided Lipschitz condition on $f$ to obviate the problem that the Nemytskii operator $u\mapsto f\circ u$
is not Lipschitz on $L^2$.
Ball and \Sengul\ then make rigorous the interpretation of these solutions as a gradient flow of $\E$ in a constant-mass subset of $L^2(\Omega)$. 
Further, they prove the $L^2$ compactness of positive orbits using monotone rearrangement and Helly's theorem,
and they improve the convergence analysis in the studies \cite{AndBall82,Novick-CohenPego91} in several ways.
They prove that solutions converge to equilibrium under a weakened nondegeneracy condition. 
For the cubic case $f(u)=u^3-u$ in particular, convergence is proved under the hypothesis 
that 
\begin{equation} \label{c:zeromean}
\int_{\Omega}u(x,0)\,d\nu(x)\neq 0 \,.
\end{equation}

A nearly contemporaneous study by Hilhorst {\em et al.} \cite{HilhorstEA15} was motivated by 
study of a singularly perturbed Allen-Cahn equation with mass conservation \cite{RubStern92}.
These authors studied existence and uniqueness of solutions of \eqref{eq:ODE_nonlocal} taking values in $L^\infty(\Omega)$ for multistable nonlinearities \cite[Theorem 1.4]{HilhorstEA15}, and proved stabilization for bistable nonlinearities $f$ when the initial data have no flat portions \cite[Theorem 1.6]{HilhorstEA15}, having the property that all level sets $\{x\in\Omega\mid u(x,0)=c\}$ have measure zero. Based on the asymptotic behavior of solution to of the nonlocal ODE, they study the generation of interfaces for solutions of the mass-conserved Allen-Cahn equation \cite{HilhorstEA20}.

\subsection{Discussion and plan}

Gradient flows are generally important in many areas in mathematics, including in optimization for purposes such as 
training artificial neural networks \cite{RobMon51,BCFN18,FoNiQu06,EJRW23} 
and improving methods of statistical sampling~\cite{Garcia-TrillosEA}.
The \Loja\ gradient estimates provide a powerful tool to conclude convergence of 
finite-dimensional gradient flows with analytic and also nonsmooth subanalytic nonlinearities~\cite{BolteEA2007}.
Simon's extensions have allowed the handling of some infinite-dimensional flows, 
particularly for partial differential equations of parabolic type in which 
the infinite-dimensional dynamics can be slaved to some finite-dimensional part
by a kind of Lyapunov-Schmidt reduction~\cite{HarauxJendoubi2015}. More recently, \Loja -type inequalities have been extended to general metric spaces \cite{HauMaz19}. 

In light of these strong results from gradient-estimate theory,
our counterexamples for solutions of \eqref{eq:ODE_nonlocal} are puzzling
insofar as they  work for the simplest kinds of probability spaces and nonlinearities.
For example, non-convergent solutions can be found on the one-dimensional domain $\Omega=[0,1]$
which are monotone in $x$ and have compact trajectories in $L^2$,
and which have finite-dimensional (actually one-dimensional) $\omega$-limit sets.
Moreover, the nonlinear function $f$ can be polynomial (cubic), both as a real function
and as a Nemytskii operator on $B(\Omega)$
(although the latter is not even once Fr\'echet differentiable on the space $L^2(\Omega)$).

So despite the rather benign nature of nonlocally coupled \blue differential \nc equations
from the point of view of nonlinear analysis,
having a very regular nonlinear structure and having essentially finite-dimensional long-time dynamics
appears insufficient to ensure gradient-flow convergence.  
For finite-dimensional flows, solutions converge, but our constructions indicate that
the rate of convergence can be arbitrarily slow, even for fixed nonlinearity and fixed dimension as small as 3.
The appearance of arbitrarily slow rates of convergence is a curious phenomenon, in fact---it happens by perturbation 
from a situation in which the rate of convergence is $O(1)$ and a \Loja\ inequality applies.

Our non-convergent examples are all non-generic and highly unstable.
To emphasize how delicate non-convergence has to be for the cubic nonlinearity,
we present the following criterion that is necessary (but far from sufficient) for non-convergence, 
which shows that non-convergence is far more unlikely to arrange than the 
codimension-1 necessary condition $\int_\Omega u=0$ from \eqref{c:zeromean} might suggest.
 \begin{proposition}[Unstable nature of non-convergence]\label{prop:unstable_counterex}
             Let $f(u)=u^3-u$, and suppose $u(\cdot,t)$ is a bounded solution of \eqref{eq:ODE_nonlocal} that fails
             to converge in $L^2$ to a limit as $t\to\infty$.  Then $\int_\Omega u(x,0)\,d\nu(x) =0$, 
             and moreover, there exists $c$ such that the three sets, consisting  
             of all $x\in\Omega$ where $u(x,0)=c$,  where $u(x,0)>c$, and where $u(x,0)<c$ respectively,
             each have measure exactly equal to  $\frac13.$
     \end{proposition}

The plan of this paper is as follows. 
We develop a few basic properties of solutions of \eqref{eq:ODE_nonlocal} in Section~\ref{s:basic},
regarding well-posedness, the relative preservation of order at different values of $x$,
and invariant sets for solutions (a kind of maximum principle).
In Section~\ref{s:finite-d} we re-prove long-time convergence for solutions with finite range,
in a simpler way than in \cite{Pego92} using gradient estimates. 
Our construction of non-convergent solutions for piecewise-linear bistable $f$ appears in 
Section~\ref{s:nonconvergence-PL}.  Subsection~\ref{ss:gradPL} contains an $L^2$ gradient inequality 
that is valid in this case (Lemma~\ref{lem:gradE-PL}) which is curiously similar to 
the one used to prove convergence in the finite-range case  with arbitrary nonlinearity (Lemma~\ref{lem:gradE}).
In Section~\ref{s:cubic} we construct non-convergent examples for cubic $f$, 
and also complete the proof of Proposition~\ref{prop:unstable_counterex}. 

Finally we discuss in Section \ref{sec:sensitivity} a phenomenon of instability of convergence rates under perturbation around degenerate equilibria.  For suitable three-valued initial data, parameter perturbations of order $O(\eps)$ leads to slow exponential convergence at rate $O(\eps)$, whereas a rate of order $O(1)$ is {\purp guaranteed by the gradient inequality in Lemma~\ref{lem:gradE} \nc} when $\eps=0$. 

\section{Basic properties of solutions}\label{s:basic}

We begin our analysis with a brief discussion of the well-posedness of the initial value problem for
\eqref{eq:ODE_nonlocal}, and some basic properties that solutions have
regarding preservation of order and positively invariant sets.

We choose to work with solutions taking values $u(\cdot,t)$ in the space of bounded measurable functions $B(\Omega)$,
as it is convenient to interpret them as pointwise satisfying the \blue nonlocal \nc differential equation in \eqref{eq:ODE_nonlocal},
without having to take the trouble of selecting representatives from equivalence classes 
as was done in \cite{Novick-CohenPego91} for elements of $C([0,T],L^\infty(\Omega))$.
Local-time well-posedness (existence, uniqueness, and continuous dependence on initial data) follows by 
the standard Picard iteration method.  
This use of $B(\Omega)$ makes well-posedness and the study of pointwise properties rather easy, 
as solutions $u(x,t)$ are $C^1$ in $t$ for every $x$, but some other things become more difficult. 
E.g., even in case $\Omega=[0,1]$ with Lebesgue measure, it does not seem easy to determine 
whether, say, measurable monotone reordering is possible pointwise everywhere for all initial data.

We will make considerable use of the pointwise properties that solutions enjoy according to the two following results. \blue The first lemma was established in the proof of \cite[Theorem 2]{BallSengul15}\nc. The second one is similar to results observed  in \cite[Lemma 2.5]{HilhorstEA15} \blue and \cite[Corollary 2]{BallSengul15} \nc  and 
previously for viscous diffusion equations in \cite[Proposition~2.7]{Novick-CohenPego91}.

    \begin{lemma}[Preservation of order]\label{lem:monotone}
        Let $u$ solve the nonlocal ODE \eqref{eq:ODE_nonlocal}. If $u(x,0)< u(y,0)$, then for all $t>0$ we have $u(x,t)< u(y,t)$. Further, equality is also preserved.
    \end{lemma}
    
    \begin{proof} This is a simple consequence of the fact that if we regard $\bar f(t)$ as given,
    then $u(x,t)$ and $u(y,t)$ satisfy the same scalar ODE with locally Lipschitz nonlinearity.
    \end{proof}

We call a set $S\subset\R$ {\em positively invariant} for \eqref{eq:ODE_nonlocal} if the condition
$u(x,0)\in S$ for all $x\in\Omega$ implies that $u(x,t)\in S$ for all $x\in\Omega$ and $t>0$.
For a given solution $u$,  we call a set $\hat S\subset \R$ {\em pointwise stable}
if  $u(\hat x,0)\in \hat S$ implies $u(\hat x,t)\in\hat S$ for all $t>0$,
for any (particular) $\hat x\in\Omega$. 
\begin{lemma} \label{lem:invar_rgn}
\begin{itemize}
    \item[(i)] (Positively invariant sets)  Let $[a,b]$ be a closed interval such that 
       \[f(a)\leq f(s) \leq f(b) \quad\text{ for all } s\in [a,b].\]
       Then $[a,b]$ is positively invariant.
       \item[(ii)] (Pointwise stable subsets) If further $[\hat a,\hat b]\subset [a,b]$ 
       with $f(\hat a)=f(a)$ and $f(\hat b)=f(b)$,
       then $[\hat a,\hat b]$ is pointwise stable for any solution with $u(x,0)\in [a,b]$ for all $x\in\Omega$.
  \end{itemize}
\end{lemma}

\begin{proof}
Let $u(0,x)\in[a,b]$ for all $x\in\Omega$. If $f(u(\cdot,0))$ is a.e.~constant,
then $\bar f(0)\in[f(a),f(b)]$ is this same constant.
So $u$ is at equilibrium a.e., and
trivially the invariance properties in parts (i) and (ii) hold. 

Suppose $f(u(\cdot,0))$ is not a.e.~constant. Then $u$ is not a.e.~at equilibrium,
and $\bar f(0)\in (f(a),f(b))$. Define  
\[
 t_\ast = \inf\{t>0: \bar f(t)\in\{f(a),f(b)\}\}.
 \]
This is the first exit time of $\bar f(t)$ from the interval $(f(a),f(b))$.
By continuity of $\bar f$, we know $t_\ast>0$. 

Next note that for any $x\in\Omega$ and  $t\in[0,t_\ast)$,
\[
-f(u(x,t))+ f(a) < -f(u(x,t))+\bar f(t) = 
\D_t u(x,t) < -f(u(x,t))+f(b).
\]
By consequence, $\D_t u(x,t)$ is positive if $u(x,t)=a$ (or $\hat a$) and 
negative if $u(x,t)=b$ (or $\hat b$).  
It follows $u(x,t)\in (a,b)$ for all $t\in(0,t_\ast)$, and all $x$.
Moreover if $u(x,0)$ is in $[\hat a,\hat b]$ then $u(x,t)$ remains there for all $t\in[0,t_\ast)$.

Now we claim $t_\ast=\infty$.  
If $t_\ast<\infty$, then by continuity $u(x,t_\ast)\in[a,b]$ and 
$f(a)\le f(u(x,t_\ast))\le f(b)$ for all $x$.  
But then 
$f(u(x,t_\ast))$ must a.e.~equal $f(a)$ if  $\bar f(t_\ast)=f(a)$, and 
must a.e.~equal $f(b)$ if  $\bar f(t_\ast)=f(b)$. 
This contradicts our hypothesis and establishes $t_\ast=\infty$. 
The invariance properties follow. 
\end{proof}

By this result, if  $[a,b]$ is an interval with the property stated
and the initial data $u(x,0)$ belong to this interval, then the solution to \eqref{eq:ODE_nonlocal} exists globally 
with $u(x,t)\in[a,b]$ for all $t\ge0$ and all $x\in\Omega$ \blue (cf. \cite[Section 3]{BallSengul15}).\nc

\section{The case of finite range: convergence via gradient inequalities}\label{s:finite-d}

Let $u$ take finitely many values $u_j$ on sets $\Omega_j\subset\Omega$ 
of measure $\mu_j$, $j=1,\ldots,N$, with $\sum_k \mu_k = 1$.
Our equation is then equivalent to the following system in $\R^N$:
\begin{equation}\label{eq:uRN}
 \frac{d}{dt}u_j(t) = -f(u_j(t)) + \bar f(t), \quad \blue j=1,\ldots,N,\nc
 \qquad 
 \bar f(t) = \sum_k \mu_k f(u_k(t)).
\end{equation}
We define a reduced energy for vectors $\uu=(u_j)\in\R^N$ by \purp restricting $\E$ to functions $u = \sum_k u_k\one_{\Omega_k}$,
writing \nc
\[
 E(\uu) \purp = \E\left(\sum_k u_k\one_{\Omega_k}\right) \nc = \sum_k \mu_k F(u_k) . 
\]
With respect to the reduced \purp $L^2$-inner product \nc $\ip{\uu,\vv} = \sum_k \mu_ku_kv_k$,
\purp we obtain \nc the gradient $\nabla E(\uu) = (f(u_j))$, and \blue we may write \eqref{eq:uRN} in
the vector form \nc 
\[
\frac{d\uu}{dt} = -Q\nabla E(\uu(t)), 
\qquad
Q\vv = \vv - \one\ip{\one,\vv} = \left(v_j - \sum_k \mu_kv_k\right)  ,
\qquad \purp \one = (1,\ldots,1).\nc
\]
Here $Q$ is the orthogonal projection on the subspace \blue where $\ip{\one,v}=\sum_k \mu_k v_k = 0$.\nc

Recall that we assume $f$ is locally Lipschitz and piecewise $C^1$. Our goal in this
section is to provide a simplified proof of the following theorem from \cite{Pego92}.
\begin{theorem}\label{t:finitedim}
    If $\uu\colon[0,\infty)\to\R^N$ is a bounded solution of \eqref{eq:uRN},
    then $\lim_{t\to\infty} \uu(t)$ exists.
\end{theorem}

Our simplified proof avoids a spectral analysis of curves of equilibria of \eqref{eq:uRN}
and the use of the Hale-Massat theorem. Instead we rely on the gradient inequality
contained in the following lemma.  Its proof involves a simple Taylor approximation
argument near \purp curves of \nc ``regular equilibria,'' which stands
in contrast to general \Loja\ inequalities valid near arbitrary equilibria 
for energies that are analytic, semi-algebraic, or more generally
definable in an o-minimal structure~\cite{AttouchEA2013}.

We recall as in \cite{Pego92} that by Sard's theorem, the set of {\em regular values} of $f$ 
in any bounded interval of $\R$ is open and dense. 
If $\hat s$ is a regular value of $f$, 
then  the equation $f(z)=s$ has a finite number of solutions 
$z_i(s)$ at which  $f'(z_i(s))\ne0$,  
for all $s$ in some neighborhood $\hat J$ of $\hat s$. 
We will call $\hat\uu\in\R^N$ a {\em regular equilibrium} for  \eqref{eq:uRN}
if $\hat s=f(\hat u_j)$ is independent of $j$ and is a regular value of $f$.
In this case, then for each $j$ there exists $i(j)$ such that $\hat u_j=z_{i(j)}(\hat s)$.
We define $\phib(s)=(z_{i(j)}(s))$ for $s\in\hat J$; then 
$s\mapsto \phib(s)$ is a curve of regular equilibria and 
$\phib(\hat s)=\hat\uu$. 
\begin{lemma}\label{lem:gradE}
 Let $\hat \uu\in\R^N$ be a regular equilibrium for \eqref{eq:uRN} as above.
 Then in some neighborhood $\calN$ of $\hat \uu$, all equilibria of \eqref{eq:uRN} have the form 
 $\phib(s)$ for some $s\in \hat J$, and 
 {\purp moreover: \nc}
 \begin{itemize}
\item[(i)] all states $\uu\in\calN$ satisfy the gradient inequality
 \[
 c|E(\uu)-E(\phib(s))- s\ip{\one,\uu-\phib(s)}| \le \|Q\nabla E(\uu)\|^2\,,
 \quad 
s=\sum_j \mu_j f(u_j), 
 \]
for some constant $c>0$ independent of $\uu$. 
\item[(ii)]
{\purp 
If $\ip{\one,\phib(s)-\hat\uu}=0$ for all $s\in \hat J$, then $E(\phib(s))\equiv E(\hat \uu)$ and
for all $\uu\in\calN$ with $\ip{\one,\uu-\hat \uu}=0$ we have \nc}
\[
{\purp c|E(\uu)-E(\hat \uu)| \le \|Q\nabla E(\uu)\|^2\,. \nc}
\]
\end{itemize}
\end{lemma}
\begin{remark}\label{rmk:LojaN} 
    The inequality in part (ii) of this Lemma can be interpreted as a \Loja\ inequality in the constrained-mean hypersurface 
    $\M=\{\uu\in\R^N: \sum_j\mu_j(u_j-\hat u_j)=0\}$, since the gradient of $E$ restricted to this surface
    can be interpreted as the projection $Q\nabla E$ on the tangent space. 
    \purp The proof we give below is simple and direct.
    An alternative proof could be given by showing that the curve of equilibria $\phib(s)$ 
    satisfies certain nondegeneracy properties within the hypersurface $\M$, and 
    applying, say, Theorem~2.1 of~\cite{HarauxJendoubi2007}, 
    or the reduction methods of Simon~\cite{Simon96} or Chill~\cite{Chill2003} mentioned in the introduction.
    The required nondegeneracy properties are somewhat involved to establish, though, 
    due to the fact that the eigenvalue $\lambda=0$ of the full Jacobian matrix $\D Q\nabla E/\D\uu$
   at $\hat\uu$ in $\R^N$ is not algebraically simple~\cite[Lemma~2]{Pego92}.\nc
\end{remark}

\begin{proof}
  For any equilibrium $\uueq$ in a small enough neighborhood $\calN$ of $\hat\uu$, 
  $s=f(\ueq_j)$ is independent of $j$ 
  and near $\hat s$, so necessarily $\ueq_j=z_{i(j)}(s)$ by  the inverse function theorem.
  Taking $\calN$ smaller if necessary, for any $\uu\in\calN$ we may let 
  \[
   s= \sum_j\mu_jf(u_j), \qquad \vv = \uu - \phib(s),
  \]
and we may find constants $0<\underline\lambda <\overline\lambda<\infty$ such that 
  $\underline\lambda<|f'(u_j)|<\overline{\lambda}$ for all $\uu\in\calN$ and all $j$.
  By Taylor's theorem we may write
  \begin{align}
      F(u_j) &= F(\phi_j(s))+ f(\phi_j(s))v_j+ \frac12\ell_j(\uu)v_j^2 \,,
      \qquad
      f(u_j) = f(\phi_j(s))+ \hat\ell_j(\uu)v_j \,,
  \end{align}
  where  
  \begin{equation}
      \ell_j(\uu) = 2\int_0^1 f'(\phi_j(s)+rv)(1-r)\,dr\,,
      \qquad
      \hat \ell_j(\uu) = \int_0^1 f'(\phi_j(s)+rv)\,dr\,.
  \end{equation} 
  The bounds $\underline\lambda<|\ell_j(\uu)|,|\hat \ell_j(\uu)|<\overline\lambda$
  hold for all $\uu\in\calN$.   Then since $s=f(\phi_j(s))$ we have
  \begin{align*}
     E(\uu)-E(\phib(s)) &= \sum_j\mu_j \Bigl( F(u_j)-F(\phi_j(s)) \Bigr)
     = \sum_j \mu_j \Bigl( f(\phi_j(s))v_j + \frac12\ell_j(\uu)v_j^2 \Bigr)
     \\ &= s\sum_j \mu_j v_j + 
     \frac12 \sum_j \mu_j \ell_j(\uu) v_j^2 \,.
  \end{align*}
  Since also $s=\sum_k \mu_k f(u_k)$,  we find 
  \begin{align*}
      Q\nabla E(\uu)_j &= 
      f(u_j)- \sum_k \mu_k f(u_k) = 
      f(u_j)-f(\phi_j(s)) 
    =  \hat\ell_j(\uu) v_j \,,
  \end{align*}
  hence $\|Q\nabla E(\uu)\|^2  = \sum_j \mu_j \hat\ell_j(\uu)^2 v_j^2 $.
Evidently we have the estimates
 \[
\Bigl|\sum_j \mu_j \ell_j(\uu) v_j^2 \Bigr| \  \le 
     \ \overline\lambda \sum_j \mu_j v_j^2
    \ \le \ \frac{\overline\lambda}{\underline\lambda^2} \sum_j \mu_j \hat\ell_j(\uu)^2 v_j^2 \,,
 \]
whence the result claimed in \purp part (i) of \nc the Lemma follows with $c=\frac12\underline\lambda^2/\overline\lambda$.

\purp 
If $\ip{\one,\phib(s)}$ is constant in $s$, then because $f(\phi_k(s))=s$, 
\[
\frac{d}{ds} E(\phib(s)) = \frac{d}{ds} \sum_k \mu_k F(\phi_k(s)) = \sum_k \mu_k f(\phi_k(s))\phi_k'(s) = s \frac{d}{ds}\ip{\one,\phib(s)} = 0.
\]
Hence $E(\phib(s))\equiv E(\phib(\hat s))=E(\hat\uu)$, and the rest of part (ii) follows from part (i).\nc
\end{proof}

The next (and main) step in the proof of Theorem~\ref{t:finitedim} is to show that $\bar f(t)$ converges. 
This is as in \cite{Pego92}, but now the proof is much simpler. 

\begin{lemma}\label{lem:barf}
    If $\uu\colon\R^N\times[0,\infty)$ is a bounded solution of \eqref{eq:uRN},
    then $\lim_{t\to\infty}\bar f(t)$ exists. 
\end{lemma}

\begin{proof}
    Suppose not. Then the interval $(\liminf\bar f,\limsup\bar f)$  is nonempty and strictly contains some interval $\hat J$
    of regular values of $f$, by Sard's theorem as above. Fixing some $\hat s\in\hat J$, using the compactness of
    the orbit we can find a sequence $t_n\to\infty$ such that $\bar f(t_n)=\hat s$ and $\uu(t_n)$ converges
    to some regular equilibrium $\hat \uu\in \omega(\uu)$.  Then because $\omega(\uu)$ is connected and 
    $\hat\uu$ cannot be isolated in $\omega(\uu)$, 
    by taking $\hat J$ smaller and on one side of $\hat s$ if necessary, the curve of equilibria $\{\phib(s):s\in\hat J\}$ 
    provided by the Lemma will be entirely contained in $\omega(\uu)$. 

    By consequence, we infer that for all $s\in\hat J$,
    \begin{equation}\label{e:s_id}
        E(\phib(s))=E_* \quad\text{and}\quad \sum_j\mu_j\phi_j(s) = c_0 = \sum_j \mu_j u_j(t) \,,
    \end{equation}
    where $E_* = \lim_{t\to\infty}E(\uu(t))$. By the result of the Lemma, then, we have 
    \begin{equation}\label{eq:LojaN}
     0< \hat c \sqrt{ E(\uu(t))-E_*}  \le \|Q\nabla E(\uu(t))\|  
    \end{equation}
    whenever $\uu(t)\in\calN$, a small enough neighborhood of $\hat u$. 
    But then, by the classic argument of \Loja,
    and because $Q=Q^2$ is self-adjoint,
    \begin{align}
        - \frac{d}{dt}
        \sqrt{E(\uu(t))-E_*} 
        &= \frac{\ip{\nabla E(\uu(t)),Q^2\nabla E(\uu(t))}}
        {2 \sqrt{E(\uu(t))-E_*} }
        = \frac{\|Q\nabla E(\uu(t))\|\|\D_t \uu\|}
        {2 \sqrt{E(\uu(t))-E_*} }
        \ge \frac{\hat c}{2} \|\D_t \uu\|.
    \end{align}
    On any interval $[t_n,T]$ on which $\uu(t)\in\calN$  it follows  
    \[
    \|\uu(T)-\uu(t_n)\| \le \int_{t_n}^T \|\D_t \uu(\tau)\|\,d\tau \le C \sqrt{E(\uu(t_n)-E_*}.
    \]
    For large enough $n$, the right-hand side becomes arbitrarily small 
    and it follows $\uu(t)$ remains inside $\calN$ for all $t\ge t_n$. This 
    implies $(\liminf \bar f,\limsup\bar f)\subset \hat J$,  a contradiction.
    Hence $\lim_{t\to\infty}\bar f(t)$ exists.
\end{proof}

The remainder of the proof of Theorem~\ref{t:finitedim} goes as in \cite{Pego92}, in principle.
However, the proof in that paper appears to have a gap (in Lemma 3 in particular), so we
provide a full corrected proof here for the convenience of the reader.

\begin{proof}[Proof of Theorem~\ref{t:finitedim}]
Suppose for contradiction that some bounded solution $\uu$ of \eqref{eq:uRN} fails to converge. 
Then $a_j<b_j$ for some $j$, where
\[
a_j = \liminf u_j(t), \quad b_j = \limsup u_j(t) \,, \quad j=1,\ldots,N.
\]
Due to Lemma~\ref{lem:barf}, by adding a constant to $f$ we may assume $\bar f(t)\to0$ as $t\to\infty$.
By considering times $t_{n,j}\to\infty$ such that $u_j(t_{n,j})$ takes given limits inside $(a_j,b_j)$,
we infer $f(v)=0$ for all $v\in \bigcup_j [a_j,b_j]$.  

The idea of the remainder of the proof is that mass conservation $\sum_k \mu_k u_k(t)=c_0$ must become
violated, due to the synchrony implied by the  equations $\D_tu_j=\bar f(t)=\D_t u_k$ 
which must hold whenever $u_j$ and $u_k$ are respectively inside 
any nonempty open intervals $(a_j,b_j)$, $(a_k,b_k)$.

Select a point $\vv$ in $\omega(\uu)$ such that $v_j\in(a_j,b_j)$
for $j$ in some {\em maximal} set $S$ of indices.   
With the notation $B(x,r)=[x-r,x+r]$, choose $\eps>0$ so that
$B(v_j,2\eps)\subset (a_j,b_j)$ for all $j\in S$,
and select $t_n\to\infty$ such that $\uu(t_n)\to\vv$ as $n\to\infty$
and $u_j(t_n)\in B(v_j,\eps)$ for all $j$ and $n$. 
Now fix some $i\in S$ and define
\[
T_n = \inf\{t>t_n: |u_i(t)-u_i(t_n)|>\eps\}, \quad I_n = [t_n,T_n]. 
\]
Then for all $n$, $t_n<T_n<\infty$ and $u_i(t)\in B(v_i,2\eps)$ for all $t\in I_n$.  
Moreover, for any $j\in S$, by synchrony we have 
\[
u_j(t)-u_j(t_n) = u_i(t)-u_i(t_n) \in [-\eps,\eps]
\quad\text{and}\quad  u_j(t)\in B(v_j,2\eps)
\]
{for all $t\in I_n$}.
In particular, when $t=T_n$ it follows there is a fixed sign $\sigma\in\{-1,+1\}$  such that 
\begin{equation}\label{e:sigma}
u_j(T_n)-u_j(t_n) = u_i(T_n)-u_i(t_n) = \sigma\eps .
\end{equation}
By passing to a subsequence we may presume this holds for all $n$ with $\sigma$ independent of $n$.

We claim next that for all indices $k\notin S$,  
\begin{equation} \label{e:osc}
\osc_{I_n} u_k \to 0 \quad\text{as $n\to\infty$},
\end{equation}
where $\osc$ is the oscillation---supremum minus infimum on the indicated interval.
Suppose not. Then for some $k$, $\osc_{I_n} u_k\ge\hat\eps>0$ for infinitely many $n$. 
Hence $b_k-a_k\ge\hat\eps_n$, and by continuity there exist $\tau_n\in I_n$
such that $u_k(\tau_n)=\hat v_k$ for some $\hat v_k\in(a_k,b_k)$.
We may extract a suitable subsequence such that $u_j(\tau_n)$ converges
to some $\hat v_j$ for all $j$.
In particular we find $\hat v_j\in (a_j,b_j)$ for all $j\in S\cup\{k\}$.
This contradicts the maximality of $S$. Hence \eqref{e:osc} holds.

From this it follows $u_k(T_n)-u_k(t_n)\to0$ for all $k\notin S$.
Along the appropriate subsequence then,
mass conservation together with \eqref{e:sigma} implies
\[
\sum_j \mu_jv_j = \lim\sum_j\mu_ju_j(t_n) = \lim\sum_j\mu_j u_j(T_n) = 
\sigma\eps \#S + \sum_j \mu_jv_j,
\]
where $\# S$ is the cardinality of $S$.
This contradiction implies $\uu(t)$ tends to a limit.
\end{proof}

 \section{Non-convergence: the piecewise-linear case}\label{s:nonconvergence-PL}
  In this section we describe solutions to \eqref{eq:ODE_nonlocal} that do not converge as $t\rightarrow\infty$,
for the case when $f$ is piecewise linear with $N$-shaped graph, given by
\begin{equation} \label{d:fPL}
    f(z) = \begin{cases}
        z+1 & z<-\frac12,\cr  
        -z  & |z|<\frac12,\cr
        z-1 & z>\frac12.
    \end{cases}
\end{equation} 
For $|s|<\frac12$, the equation $f(z)=s$ has the three solutions $z_l(s)=-1+s$, $z_m(s)=-s$, and $z_r(s)=1+s$.
Since $z_l+2z_m+z_r \equiv 0$, we see $f$ fails to satisfy the nondegeneracy condition of Andrews and Ball~\cite{AndBall82};
this will be crucial in our construction.
We presume the probability measure $\nu$ is nonatomic. This implies that given any countable set $(\mu_j)$ with $\sum \mu_j=1$, 
there exists a measurable partition $(\Omega_j)$ of $\Omega$ such that $\nu(\Omega_j)=\mu_j$ for all $j$. 
(This follows since $\nu$ has the ``Darboux property,'' see \cite[p.~28]{Dinculeanu67} and  \cite[p.~174(2)]{Halmos_book}.)

\subsection{Equilibria, and phase transition times}
{\em Equilibria.}
With $f$ as in \eqref{d:fPL},
equation~\eqref{eq:ODE_nonlocal} has a family of equilibria $\hat u_s$ satisfying 
$f(\hat u_s(x))\equiv s$ 
for any constant $s\in(-\half,\half)$, with 
$\hat u_s(x)= z_j(s)$ on sets $\hat\Omega_j$ of measure denoted $\hat\mu_j$ for $j=l,m,r$
to indicate the left, middle, and right phases, respectively. 
We fix the particular values 
\begin{equation}\label{d:hatmu}
\hat \mu_l=\tfrac14, \quad \hat \mu_m=\tfrac12,  \quad  \hat \mu_r=\tfrac14, 
\end{equation}
so that all these equilibria have mean zero, i.e.,
\[
\int_\Omega \hat u_s(x)\,d\nu(x) = 0, \qquad\text{independent of $s$.}   
\]
Our goal in this section is to describe a solution that has some nontrivial collection of these 
equilibria in its $\omega$-limit set (in the $L^2$ topology).

{\em Phases and transition times.}
In this section, we will only consider solutions taking values in  the interval $[a,b]=[-\frac32,\frac32]$, 
which is positively invariant according to Lemma~\ref{lem:invar_rgn}. 
For the remainder of this section we fix the values
\[
a=-\tfrac32, \quad \hat b=-\tfrac12, \quad\hat a=\tfrac12, \quad b=\tfrac32,
\]
and define left, middle, and right phase intervals respectively by
\begin{equation}
\Phi_l = [a,\hat b] ,\quad \Phi_m = (\hat b,\hat a) ,\quad \Phi_r = [\hat a,b]\,.
\end{equation}
For the solutions we consider, the left and right  phase subintervals $\Phi_l$ and $\Phi_r$ are each pointwise stable.
We define measures of sets corresponding to the left, middle, and right phases by
\begin{equation}\label{d:Aj}
\nu_j(t) = \nu(A_j(t)), \quad 
A_j(t) = \{x\in\Omega: u(x,t)\in\Phi_j\},  
\end{equation}
for each symbol $j=l, m, r$ respectively. 
Then by pointwise stability, the left and right phases $A_l(t)$ and $A_r(t)$ and their measures are nondecreasing,
while the middle phase $A_m(t)$ and its measure $\nu_m(t)$ are nonincreasing.
Consequently a {transition time} (exit time) from the middle phase exists at each point, as follows.
\begin{lemma}[Phase transition times]\label{lem:tau}
    For each $x$ with $u(x,0)\in\Phi_m$, there exists $\tau(x)\in(0,\infty]$ such that 
  \[
  u(x,t)\in \begin{cases}\Phi_m, & 0\le t<\tau(x),\cr  \Phi_l\cup \Phi_r, & t\ge\tau(x). \end{cases}
  \] 
\end{lemma}
Moreover, as long as two points $u(x,t)$ and $u(y,t)$ remain in the middle phase $\Phi_m$,
the difference grows exponentially, for we have
\[
\D_t (u(x,t)-u(y,t)) = u(x,t)-u(y,t).
\]
\begin{corollary} \label{cor:expt}
If $u(x,0),u(y,0)\in\Phi_m$,
  then for $0\le t < \tau(x)\wedge\tau(y)$ we have
\[
  u(x,t)-u(y,t) = e^t (u(x,0)-u(y,0))\,.
\]
\end{corollary}

\subsection{Mean force and heuristics}\label{ss:heuristic_PL}
{\em Evolution of mean force.} For the piecewise-linear nonlinearity  in \eqref{d:fPL},
it happens that $\bar f(t)$ evolves in a strikingly simple way. 
Due to the fact that 
\begin{equation}
f'(u) = \begin{cases}
 1 & u<-\half\,,\cr
 -1 & |u| <\half\,,\cr
 1 & u>\half \,,
\end{cases}
\qquad
 f'(u)f(u) = \begin{cases}
  u+1 & u< -\half\,,\cr
  u+0 & |u|<\frac12\,,\cr
  u-1 & u> \half\,,
 \end{cases}
\end{equation}
and $\bar f(t)$ is Lipschitz, hence differentiable a.e., 
we find using \eqref{d:Aj} that  with $\bar u = \int_\Omega u\,d\nu$,
for a.e.~$t$,
\begin{align}\label{eq:dtbarf}
\frac{d}{dt} \bar f(t) &= \int_\Omega f'(u)(-f(u)+\bar f(t))\,dx = 
-(\bar u +\nu_l-\nu_r) + (\nu_l-\nu_m+\nu_r)\bar f(t).
\end{align}

{\em Heuristics.} We can now explain the main idea behind our examples of non-convergence,
by describing a simple calculation that shows
how tiny perturbations from certain (always unstable) degenerate equilibria can produce
slow, but eventually large, changes in $\bar f(t)$. 
We will consider solutions with mean $\bar u=0$.
Desiring some equilibrium $\hat u$ as above to be in the $\omega$-limit set,
$\nu_m(t)$, the measure of the middle phase, should approach $\hat\mu_m=\frac12$ from above. 
Thus we will perturb by moving small bits of the (stable) left and right phases to be in the (unstable) middle phase,
close to but not exactly at the same value as $\hat u$ takes.  

Imagine then that the  initial data takes values near $-1+s$, $-s$, $1+s$
on sets of measure 
\begin{equation}\label{eq:epslr}
\nu_j = \hat\mu_j-\eps_j\,,
\end{equation}
for each symbol $j=l,m,r$ respectively, 
with  $\eps_l,\eps_r>0$ small and $\eps_m = -\eps_l-\eps_r$.
Suppose no phase changes occur over some interval of time 
during which the measures $\nu_j$ do not change.
Then during this time interval, \eqref{eq:dtbarf} becomes
\begin{equation}\label{eq:dtbarf2}
\frac{d}{dt} \bar f(t)  = \eps_l-\eps_r -2(\eps_l+\eps_r)\bar f(t) .
\end{equation}
Regardless of what the original value of $s$ was, $\bar f(t)$ is now forced to drift
toward a particular equilibrium value determined by $\eps_l$ and $\eps_r$, namely
\begin{equation}\label{eq:barfeq}
\barfeq = \half \frac{\eps_l-\eps_r}{\eps_l+\eps_r} \ \in\left(-\half,\half\right).
\end{equation}
This value can be of order 1 no matter how small $\eps_l$, $\eps_r$ are.

Now the idea to obtain persistent oscillations is to use the exponential 
growth rate of perturbations in the (unstable) middle phase to arrange that 
small bits of that phase 
will change alternately to the (stable) left and right phases.
The time gaps between these changes should be large enough so that
$\bar f(t)$ is attracted near the prevailing value of $\barfeq$,
and the pattern of changes should cause $\eps_l$, $\eps_r$ to alternately decrease
in a way that forces the value of $\barfeq$ to
alternately drift toward distinctly different values.
We will show this can be done infinitely often, 
with the implication that $\bar f(t)$ will fail to converge
as $t\to\infty$, and the same for $u(\cdot,t)$.

\subsection{Initial data and main result}
With suitable initial data specified as follows, we can ensure that $\bar f(t)$ fails to converge.
We consider initial data taking infinitely many values, of the form 
$u(x,0)=v_0(x)-\bar v_0$ so that $\bar u=0$, with
\begin{equation}\label{d:ic_PL}
v_0(x) = \begin{cases}
 -1,\ 0,\ 1 &\text{in}\quad \Omega_l,\ \Omega_m,\ \Omega_r \text{\ \ respectively,}
\\[6pt] {(-1)^j\alpha_j} &  \text{in $\Omega_j$}\,, \ {j=0,1,2\ldots.} 
\end{cases}
\end{equation}
\blue Here, $(\alpha_j)_{j=0,1,2,\cdots}$ is a sequence of positive real numbers satisfying inequalities specified below\nc. We write $\mu_j=\nu(\Omega_j)$ for $j=l,m,r$ and $0,1,2,\ldots$, and assume 
\begin{equation}\label{e:mu_lmr}
\mu_l = \frac14-\sum_{j\,{\rm odd}} \mu_j \,,
\qquad\mu_m = \frac12\,, 
\qquad \mu_r = \frac14-\sum_{j\,{\rm even}} \mu_j \,.
\end{equation}

\begin{theorem}[Counterexample to convergence]\label{thm:PL}
Let $f$ be given by \eqref{d:fPL} and consider initial data for \eqref{eq:ODE_nonlocal} 
of the form $u(x,0)=v_0(x)-\bar v_0$ with $v_0$ given as above.
Let $0<\eta<1$,  and assume $0<\mu_0\le \frac{1-\eta}{4}$ 
and 
\[
\mu_j=\mu_0\eta^j \ \ (j=0,1,2,\ldots).
\]
Assume $0<\alpha_0<\frac14$, and that 
\begin{equation}\label{c:alpha1}
0<\alpha_{j+1}\le \alpha_j\mu_j \quad\text{ for $j=0,1,2,\ldots$}.
\end{equation}
Then:
\begin{itemize}
\item[(i)] The phase transition times 
$\tau_j=\tau(\Omega_j)$
satisfy $\tau_m=+\infty$ and $\tau_{j+1}>\tau_j$ for all $j\ge0$.
\item[(ii)]  If moreover for some positive sequence $\beta_j$ decreasing to $0$,
\begin{equation}\label{c:alpha2}
\alpha_{j+1}\le  \alpha_j \mu_j \beta_j^{1/\mu_j}
\end{equation}
for all $j$ sufficiently large,
then 
\[
\limsup \bar f(t)-\liminf \bar f(t) = \frac{1-\eta}{1+\eta}\,, 
\]
and as $t\to\infty$  the solution $u$ fails to converge 
in $L^p$ for any $p\in[0,\infty]$.
Its $\omega$-limit set consists
of all the equilibria $\hat u_s$ for $|s|\le \frac12\frac{1-\eta}{1+\eta}$.
\end{itemize}
\end{theorem}
Observe that 
\[
-\bar v_0 = \mu_l-\mu_r -  \sum_{j\ge0}(-1)^j\alpha_j\mu_j\,,
\qquad
\mu_l-\mu_r = \frac{\mu_0}{1+\eta}\,, 
\qquad
0< \sum_{j\ge0}(-1)^j\alpha_j\mu_j < \alpha_0\mu_0.
\]
Thus the hypotheses imply $0<-\bar v_0<\mu_0\le\frac14$ and it follows that 
$u(x,0)\in\Phi_j$ for all $x\in\Omega_j$ $j=l,m,r$. 
Moreover $\mu_l,\mu_r>0$ and $u(x,0)\in\Phi_m$ for all $x\in\Omega_j$ with $j\ge0$, since  $|(-1)^j\alpha_j-\bar v_0|<\frac12$.

\begin{remark}\label{rmk:monotone_init;pwl}
    In case $\Omega=[0,1]$ and $\nu$ is the Lebesgue measure on $[0,1]$, we can ensure the initial data are monotonically increasing by an explicit choice of the $\Omega_j$, setting 
    \[\Omega_l=\left[0,\frac14-\frac{\mu_0\eta}{1-\eta^2}\right),
    \qquad \Omega_m=\left[\frac14,\frac34\right],
    \qquad\Omega_r=\left(\frac34+\frac{\mu_0}{1-\eta^2},1\right],\]
    and 
    \[
    \Omega_j=\begin{cases}
    \displaystyle
    \left(\frac34+\frac{\mu_0\eta^{j+2}}{1-\eta^2},\frac34+\frac{\mu_0\eta^{j}}{1-\eta^2}\right] \text{ for even }j\geq 0,\\[12pt]
    \displaystyle
    \left[\frac14-\frac{\mu_0\eta^{j}}{1-\eta^2},\frac14-\frac{\mu_0\eta^{j+2}}{1-\eta^2}\right) \text{ for odd } j\geq 1.
    \end{cases}
    \qedhere
    \]
\end{remark}

\subsection{Ordering of phase transition times}

In this subsection our goal is to prove part (i) of the theorem.
The ideas for this part of the proof will also apply to the case of cubic nonlinearity 
with few changes, see Section~\ref{s:cubic} below.

To begin we set some notation. 
Let $u_j(t)$ denote the value of $u(x,t)$ for $x\in\Omega_j$, 
$j=l,m,r$ and $0,1,2,\ldots$. Noting that $u_j(0)$ lies in the middle phase
$\Phi_m$ for  $j=m$ and $0,1,2,\ldots,$  
we let $\tau_j=\tau(\Omega_j)$ denote the phase transition time for all $x\in\Omega_j$. 
For convenience we also write $\tau_{-1}=0$ and $\alpha_{-1}=1$.

First, we claim $\tau_m=\tau(\Omega_m)=+\infty$.
The proof is simple based on  preservation of order, 
the invariance of the interval $[a,b]=[-\frac32,\frac32]$, and mass conservation. 
Preservation of order (Lemma~\ref{lem:monotone}) and the invariance of $[a,b]$
implies that for all $t\ge0$,
\begin{equation}\label{eq:order2}
a\le u_j(t)<u_m(t)<u_k(t)\le b 
\quad\text{for $j$ odd or $=l$, and $k$ even or $=r$.}  
\end{equation}
Supposing $\tau_m<\infty$,  we have either $u_m(\tau_m)=\frac12$ or $-\frac12$. 
Consider the first case.  Then by mass conservation and \eqref{eq:order2},
at time $t=\tau_m$,
\begin{equation} \label{bd:ubar1}
0 = \bar u > a\left(\mu_l+ \sum_{j \rm odd}\mu_j \right) 
+ \blue u_m(\tau_m)\nc \left(\mu_m+\mu_r+\sum_{j \rm even} \mu_j \right) = 
-\frac32\cdot\frac14 + \frac12\cdot\frac34 = 0,
\end{equation}
a contradiction. A similar contradiction obtains if $u_m(\tau_m)=-\frac12.$
This proves the claim.

Because now $\min_{0\le t\le T} |u_m(t)\pm\frac12|>0$ for all $T$, and $\alpha_j\to0$ as $j\to\infty$, 
by simple continuity with respect to initial data we can infer that $\tau_j\to\infty$ as $j\to\infty$.

Next we claim the phase transition times $\tau_j$ strictly increase
with $j$ for $j\ge0$. The argument is based on a refinement of the estimates above.
By preservation of order we have 
\[
u_j(t)<u_{j+2}(t)<u_m(t)<u_{k+2}(t)<u_k(t) \quad\text{for all $j$ odd and $k$ even}.
\]
Hence $\tau_{j+2}>\tau_j$ for all $j\ge0$.  
\begin{lemma}\label{lem:umbounds}
   For any $t>0$,  let 
   \[
  j_l(t) = \min\{ j\ {\rm odd}: \tau_j\ge t\} ,\qquad
  j_r(t) = \min\{ j\ {\rm even}: \tau_j\ge t\} .
   \]

   Then we have the bounds
   \[
   u_m(t) < \frac12 - {\mu_{j_l(t)}} \,,
   \qquad u_m(t) > -\frac12 +  {\mu_{j_r(t)}}\,.
   \]
\end{lemma}
\begin{proof}
The sets of points initially in the unstable phase 
that transition into the left and right stable phases 
at time $t$ or later  have measure respectively given by 
\begin{equation} \label{d:ep_lr}
    \eps_l(t) = 
    \nu(\{x: \hat b \le u(x,t)<u_m(t)\}) 
\,,
\qquad
    \eps_r(t) = 
    \nu(\{x: u_m(t)<u(x,t)\le \hat a \}) 
\,,
\end{equation}
which here satisfy
\[
\eps_l(t)  = \hat\mu_l-\nu_l(t)= 
  \sum_{{\rm odd\,}j\ge j_l} \mu_j = \frac{\mu_{j_l}}{1-\eta^2} \,,
\quad
\eps_r(t) =  \hat\mu_r-\nu_r(t)= 
\sum_{{\rm even\,}k\ge j_r} \mu_k = \frac{\mu_{j_r}}{1-\eta^2} \,.
\]
Using these quantities we can obtain a bound on $u_m(t)$ with inequalities similar to \eqref{bd:ubar1}.
Namely, preservation of order and invariance imply
\begin{align}
 0 &= \bar u > a(\hat\mu_l-\eps_l) -\hat a\eps_l + u_m(t)(\hat\mu_m+\eps_r) + \hat a(\hat\mu_r-\eps_r) ,
    \\
 0 &= \bar u < \hat b(\hat\mu_l-\eps_l) + u_m(t)(\eps_l+ \hat\mu_m) + \hat a\eps_r + b (\hat\mu_r-\eps_r).
\end{align}
Recalling $a=-b=-\frac32$ and $\hat b=-\hat a=-\frac12$
it follows
\begin{align}
   u_m(t) &< 
   \left(\frac12+\eps_r\right)\inv\left(\frac14+\frac12\eps_r-\eps_l\right)
   = \frac12 - \frac{2\eps_l}{1+2\eps_r} 
   = \frac12 - \frac{2\mu_{j_l}}{1-\eta^2+2\mu_{j_r}} ,
   \\
   u_m(t) &> 
   \left(\frac12+\eps_l\right)\inv\left(-\frac14-\frac12\eps_l+\eps_r\right)
   = -\frac12 + \frac{2\eps_r}{1+2\eps_l} 
   = -\frac12 + \frac{2\mu_{j_r}}{1-\eta^2+2\mu_{j_l}} .
\end{align}
Since $2>1-\eta^2+2\mu_j$ for all $j$, 
this finishes the proof of the lemma.
\end{proof}

Now we finish the proof of part (i) of the theorem, considering even and odd cases separately.
Let $k\ge0$ be even. Then by Corollary~\ref{cor:expt}, at $t=\tau_k$ we have 
\[
\alpha_k e^{\tau_k} = u_k(t)-u_m(t) = \frac12-u_m(t) < 1.
\]
We claim $\tau_{k+1}>\tau_k$. If not, then for $t=\tau_{k+1}$ we have $u_{k+1}(t)=-\frac12$,
$j_l(t)=k+1$ and $j_r(t)\le k$,
hence by \blue Lemma \ref{lem:umbounds} \nc and Corollary~\ref{cor:expt} we have
\begin{align}
    \mu_k\le \mu_{j_r(t)} < u_m(\tau_{k+1})+\frac12 
    = \alpha_{k+1} e^{\tau_{k+1}}
    \le \alpha_{k+1} e^{\tau_{k}} < \frac{\alpha_{k+1}}{\alpha_k}.
\end{align}
This contradicts \eqref{c:alpha1}, proving $\tau_{k+1}>\tau_k$.
Similarly, for $j$ odd, at $t=\tau_j$ we have 
$ \alpha_j e^{\tau_j} = u_m(t) + \frac12 <1,$
and if $\tau_{j+1}\le \tau_j$ then for $t=\tau_{j+1}$ we have $u_{j+1}(t)=\frac12$, 
$j_r(t)=j+1$ and $j_l(t)\le j$, hence
\[
    \mu_j < \frac12-u_m(\tau_{j+1})   = \alpha_{j+1} e^{\tau_{j+1}} < \frac{\alpha_{j+1}}{\alpha_j}.
\]
Thus we conclude $\tau_{j+1}>\tau_j$ for all $j=0,1,2,\ldots$. 
This finishes the proof of part (i) of the Theorem.

\subsection{Proof of non-convergence}

For times $t$ in any interval $(\tau_{k-1},\tau_{k})$ between transition times ($k\ge0$),
$\bar f(t)$ evolves according to \eqref{eq:dtbarf},
which can be written using \eqref{d:ep_lr} as 
\begin{equation}\label{eq:dtbarf3}
\frac{d}{dt} \bar f(t)  = 
-2\eps_k(\bar f(t)-\barfeq_k),
\qquad
\eps_k=\eps_l+\eps_r =  \frac{\mu_{k}}{1-\eta},
\quad \barfeq_k = \frac{(-1)^{k-1}}{2} \frac{1-\eta}{1+\eta} ,
\end{equation}
because for $k$ even we have $j_l(t)=k+1$, $j_r(t)=k$, and for $k$ odd we have $j_l(t)=k$, $j_r(t)=k+1$.
Then because $\eps_k>\mu_k$,
\begin{equation}
|\bar f(\tau_k)-\barfeq_k| =
|\bar f(\tau_{k-1})-\barfeq_k| e^{-2\eps_k(\tau_k-\tau_{k-1})}
< e^{-\mu_k(\tau_k-\tau_{k-1})}\,.
\end{equation}
Since 
$\alpha_ke^{\tau_k} = \frac12 - u_m(\tau_k)$  for $k$ even
and  $\alpha_ke^{\tau_k} = u_m(\tau_k)+\frac12 $ for $k$ odd,
by Lemma~\ref{lem:umbounds} we infer  
$\alpha_k e^{\tau_k} > \mu_{k+1}$ and 
$ \alpha_{k-1}e^{\tau_{k-1}} \le 1$ in both cases. Hence for $k$ sufficiently large,
\[
e^{\tau_k-\tau_{k-1}} >  \frac{\alpha_{k-1}\mu_{k+1}}{\alpha_k}
> \eta^2 \beta_{k-1}^{-1/\mu_{k-1}}\,,
\]
due to the hypothesis \eqref{c:alpha2}, and it follows
\[
\mu_k(\tau_k-\tau_{k-1}) > \mu_k\log\eta^2 - \eta\log\beta_{k-1} \to \infty
\quad\text{as $k\to\infty$.  }
\]
Thus $|\bar f(\tau_k)-\barfeq_k|\to0$ as $k\to\infty$,
and this entails the result in part (ii) of the Theorem.

\subsection{A gradient inequality, insufficient for convergence}\label{ss:gradPL}
It is curious to note that for the piecewise-linear nonlinearity in \eqref{d:fPL}, a gradient inequality 
holds that is very similar to the one from Lemma~\ref{lem:gradE} that holds 
generally in the case of finite range. 
\begin{lemma}\label{lem:gradE-PL}
    Suppose $u\in B(\Omega)$ takes values $u(x)\in\Phi_j$ for $x\in\Omega_j$ for $j=l,m,r$, where 
    $\Omega_l\cup\Omega_m\cup\Omega_r=\Omega$.  Let $s=\overline{f(u)}=\int_\Omega f(u)\,d\nu$ and 
    let $\phi(x)=-1+s, -s, 1+s$ in $\Omega_l$, $\Omega_m$, $\Omega_r$ respectively. 
    Then
    \[
    \left|E(u)-E(\phi)-s\int_\Omega (u-\phi)\,d\nu \right| \le \frac12 \int_\Omega |f(u)-\overline{f(u)}|^2\,d\nu.
    \]
\end{lemma}
\begin{proof}
Since $u(x)$ and $\phi(x)$ belong to the same phase for all $x$, Taylor expansion
of the piecewise-quadratic primitive $F$ of $f$, and the facts that $f'=\pm1$ 
and $f(\phi(x))=s=\overline{f(u)}$ everywhere, yield 
\[
F(u)-F(\phi) -s(u-\phi) = (f(\phi)-s)(u-\phi) + \tfrac12 f'(\phi)(u-\phi)^2
= \pm \tfrac12(u-\phi)^2
\]
and 
\[
|f(u)-\overline{f(u)}|^2 = |f(u)-f(\phi)|^2 = |u-\phi|^2.
\]
Upon integration, the Lemma follows.
\end{proof}

The difference with the finite range case is that the equilibrium states $\phi$ in this Lemma
are chosen with values $\phi(x)$ in the same phase as $u(x)$ at each point, but here this means 
$\phi$ may not have the same average as $u$, and may never be in the $L^2$ $\omega$-limit set of the solution.

         \section{Non-convergence for a cubic nonlinearity}\label{s:cubic}

In order to demonstrate that the possibility of non-convergence of solutions of \eqref{eq:ODE_nonlocal}
is not due to any lack of analyticity of the nonlinear function $f$, we extend our analysis
from the previous section to deal with the case that $f$ is cubic and nonmonotone, fixing 
\begin{equation}\label{d:cubic_f}
    f(u) = u^3 -u .
\end{equation}
It will be evident that our analysis can extend to other nonlinearities with $N$-shaped graph
that admit a linear relation between distinct roots of $f(z)=s$, but we fix $f$ in the form
\eqref{d:cubic_f} for simplicity.

Moreover, to show that non-convergence is not restricted to solutions
having countable range or limited regularity, we allow initial data 
of a more general type.  When $\Omega$ is the interval $[0,1]$ or a bounded domain in $\R^d$, 
for example, our assumptions will permit initial data and solutions
to be $C^\infty$ smooth. 

  \subsection{Phases, equilibria, and transition times}
  Our solutions will take values in the phase intervals given by 
  \begin{equation}
      \Phi_l = [a,\hat b],\quad \Phi_m = (\hat b,\hat a),\quad \Phi_r=[\hat a,b],
  \end{equation}
  with 
  \[
  a = -\frac2{\sqrt 3} \,, \quad
  \hat b = -\frac1{\sqrt 3} \,, \quad
  \hat a = \frac1{\sqrt 3} \,, \quad
  b = \frac2{\sqrt 3} \,. 
  \]
See Fig.\ref{fig:cubic}.

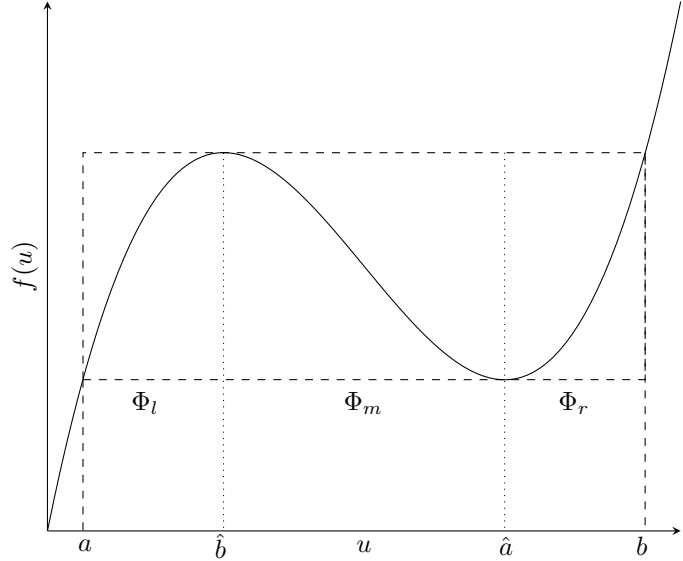
\begin{figure}
    \centering
\begin{tikzpicture}
\begin{axis}[ticks=none,
    axis lines = left,
    xlabel = \(u\),
    ylabel = {\(f(u)\)},
]
\addplot [
    domain=-1.3:1.3, 
    samples=100, 
    color=black,
] 
{x^3 - x};

\draw[dashed] (-1.154,-0.3845) rectangle (1.154,0.3845);
\draw[dashed] (-1.154,-0.3845) -- (-1.154,-1.5);
\draw[dashed] (1.154,0.3845) -- (1.154,-1.5);
\draw[dotted] (-0.577,0.3845) -- (-0.577,-1.5);
\draw[dotted] (0.577,0.3845) -- (0.577,-1.5);

\end{axis}
	  \node at (0.5,-0.2){$a$};
        \node at (2.3,-0.2){$\hat b$};
        \node at (6.1,-0.2){$\hat a$};
        \node at (7.9,-0.2){$b$};
        \node at (1.3,1.7){$\Phi_l$};
        \node at (4.2,1.7){$\Phi_m$};
        \node at (7,1.7){$\Phi_r$};
\end{tikzpicture}
    \caption{Cubic $f$ and phases.}
    \label{fig:cubic}
\end{figure}
According to Lemma~\ref{lem:invar_rgn}, the interval $[a,b]$ is invariant 
and the phase intervals $\Phi_l$, $\Phi_r$ are pointwise stable. 
Because of this, Lemma~\ref{lem:tau} holds  in this context {\em mutatis mutandi},
and transition times $\tau(x)\in(0,\infty]$ are well defined for states $u(x,t)$ 
initially in $\Phi_m$ to exit into either $\Phi_l$ or $\Phi_r$.

For each $s$ in the interval $\hat J:=(f(a),f(b))$, the equation $f(z)=s$
has a solution $z_j(s)\in\Phi_j$, $j=l,m,r$.
These three roots of the cubic equation $f(z)=s$ for $s\in \hat J$ satisfy
the trace relation $z_l(s)+z_m(s)+z_r(s)=0$. 
Our solutions will have asymptotic limits among the degenerate family 
of equilibria $\hat u_s$ taking the values $z_j(s)$ 
on sets $\hat\Omega_j$, $j=l,m,r$ of measure $\hat\mu_j$, with 
  \begin{equation}\label{e:hatmu_cubic}
      \hat\mu_l = \hat\mu_m = \hat\mu_r = \frac13.
  \end{equation}
By consequence of the trace relation,
the equilibria $\hat u_s$ all have mean $\int_\Omega \hat u_s =0$ independent of $s$.

\subsection{Heuristics: solutions with three values}\label{ss:heuristics_cubic} 
The main idea for non-convergence with the cubic nonlinearity is similar
to that for the piecewise linear case: Perturbing the equilibria $\hat u_s$
by moving a tiny amount of mass from the stable phases into the unstable phase can cause a slow
drift by a large amount. The mean force does not appear to evolve in such a
simple way as before, so we provide a different motivation.  

Consider a solution taking three values $u_j(t)\in\Phi_j$,
$j=l,m,r$,  on sets $\Omega_j$ respectively having measures
\begin{equation}
    \mu_l = \tfrac13-\eps_l\,, \quad
    \mu_m = \tfrac13+\eps_l+\eps_r\,, \quad
    \mu_r = \tfrac13-\eps_r\,,
\end{equation}
for small positive constants $\eps_l,\eps_r$, and assume that
$ 0=\bar u = \mu_l u_l + \mu_m u_m + \mu_r u_r$, 
which entails
\begin{equation}\label{eq:heu_avg}
\tfrac13 (u_l+u_m+u_r) = 
\eps_r(u_r-u_m)  -  \eps_l (u_m-u_l)  \,.
\end{equation}
Note that $u_r-u_m$ and $u_m-u_l$ are positive, and that 
$z_r(s)-z_m(s)$ and $z_m(s)-z_l(s)$ change in opposite directions as $s$ increases.
Thus we are motivated to examine the dynamics of the ``phase ratio''
\begin{equation}
    R = \frac{u_r-u_m}{u_m-u_l}
\end{equation}
as a proxy for the level of $\bar f(t)$. We find that  
\begin{align}
  \partial_t R 
  &=  - \left(\frac{f(u_r)-f(u_m)}{u_r-u_m}-\frac{f(u_m)-f(u_l)}{u_m-u_l}\right) R 
  \nonumber
  \\   &=  - (u_l+u_r+u_m)(u_r-u_l) R 
  \,,
\label{eq:Revol1}
\end{align}
since $u^3-v^3=(u^2+uv+v^2)(u-v)$. By \eqref{eq:heu_avg} this becomes
\begin{align}
  \partial_t R 
    &=   3\Bigl(\eps_l(u_m-u_l)-\eps_r (u_r-u_m)\Bigr)
    (u_r-u_l) R 
    \nonumber
    \\ &=  3\eps_r 
    \left(\frac{\eps_l}{\eps_r}-R \right)(u_r-u_l)(u_r-u_m).
\label{eq:Revol2}
\end{align}
Thus the ratio $R$ is driven to approach $\eps_l/\eps_r$ at a slow exponential rate.
Similar to the piecewise linear case, the key to obtain non-convergence will be to 
ensure that solutions behave like these three-value solutions over long time intervals, 
with the ratio $\eps_l/\eps_r$ effectively held close to constant,
but forced to change substantially infinitely many times.

\subsection{Initial data and main result}\label{ss:ic_cubic}

We will consider initial data structured in a way roughly similar to the piecewise linear case, 
but will now allow for small transition zones $\tilde\Omega_j\subset\Omega_j$. 
In case $\Omega=[0,1]$ or a bounded domain in $\R^d$, say, these transition zones
permit the initial data to be chosen
to smoothly interpolate between locally constant values in the rest of $\Omega$.
The resulting solution is then a smooth function of $x$ and $t$.  

Similar to before, we write $\mu_j = \nu(\Omega_j)$ for $j=l,m,r$ and all $j\ge0$, and suppose that
\begin{equation}\label{e:mu_lmr_cubic}
\mu_l = \frac13-\sum_{j\,{\rm odd}} \mu_j \,,
\quad\mu_m = \frac13\,, 
\quad \mu_r = \frac13-\sum_{j\,{\rm even}} \mu_j \,,
\quad \mu_j = \mu_0\eta^j\,, \quad j=0,1,2,\ldots,
\end{equation}
where $\eta>0$ is sufficiently small; it will suffice to suppose $\eta\le\frac18$.
Further, we take $\tilde\Omega_j\subset\Omega_j$ to satisfy  
\begin{equation}\label{eq:tildemuj}
\tilde\mu_j = \nu(\tilde\Omega_j)=\theta \mu_j \,,  \quad\text{ with  $\theta \in[0,\eta^2)$.}
\end{equation}
Like before, we will assume the initial data satisfy $u(x,0)=v_0(x)-\bar v_0$, but where now
\begin{equation} \label{d:ic_cubic}
v_0(x) = \begin{cases}
 -1,\ 0,\ 1 &\text{in}\quad \Omega_l,\ \Omega_m,\ \Omega_r \text{\ \ respectively,}
\\[6pt] 
                (-1)^j\alpha_j &\text{ in } \Omega_j\setminus\tilde\Omega_j\,,\quad j=0,1,2,\ldots.
            \end{cases}
\end{equation}
Furthermore, setting $\alpha_{-2}=\alpha_{-1}=1$ we require that 
\begin{equation}\label{e:tilde_order}
  0< \alpha_j\le  (-1)^j v_0(x) \le\alpha_{j-2}  \quad\text{in $\tilde\Omega_j$,}\quad j=0,1,2,\ldots.
\end{equation}
This means $v_0(x)$ is between $(-1)^j\alpha_j$ and $(-1)^j\alpha_{j-2}$ whenever $x\in\tilde\Omega_j$, 
for all $j\ge0$.
Note that we recover piecewise constant initial data by taking either 
$\theta =0$ or  $u(x,0)\equiv \alpha_j$ for all $x\in\Omega_j$.
The positive constants $\alpha_j$ must be small and decrease to zero sufficiently rapidly as described below.  

Under the mild smallness conditions 
\begin{equation}\label{c:small_cubic}
   \eta\le \tfrac18, \qquad \mu_0\le \tfrac1{10} , \qquad \alpha_0\le\tfrac12, \qquad \theta\le\eta^2,
\end{equation}
we can ensure that the initial values are in the correct phases,
with  $u(x,0)\in\Phi_j$ whenever $x\in\Omega_j$ for $j=l,m,r$, and 
$(-1)^j\alpha_j-\bar v_0\in\Phi_m$ for all $j\ge0$:
Observe that 
\begin{equation}
   \bar v_0 = \mu_r-\mu_l + \sum_{j\ge 0} 
   \left( (-1)^j\alpha_j\mu_j + 
   \int_{\tilde\Omega_j} (v_0(x)-(-1)^j \alpha_j)d\nu(x)
   \right).
\end{equation}
We have that $\mu_r-\mu_l 
=-\frac{\mu_0}{1+\eta}$, that $0<\sum_{j\ge0} (-1)^j\alpha_j\mu_j<\alpha_0\mu_0$,
and
\[
 \sum_{j\ge0} \left|   \int_{\tilde\Omega_j} (v_0(x)-(-1)^j \alpha_j)d\nu(x)\right| \le 
 \sum_{j\ge0} \theta \mu_j \alpha_{j-2} \le \frac{\eta^2\mu_0}{1-\eta} \,.
\]
Then \eqref{c:small_cubic} implies $\alpha_0<\frac1{1+\eta}$, hence 
$ |\bar v_0| \le \mu_0\left(\frac1{1+\eta} + \frac{\eta^2 }{1-\eta}\right) \le \mu_0.  $
Noting $b-1=\frac{2\sqrt3-3}3>\frac1{10}\ge\mu_0$, 
it follows  that $-1+\bar v_0\in\Phi_l$, that $1+\bar v_0\in\Phi_r$, and that 
$\alpha_0+|\bar v_0|\le \frac6{10}<\hat a$. 
This will ensure all the initial values are in the correct phases as stated.

Our main result in this section may now be stated as follows.

\begin{theorem}[Non-convergence with cubic $f$]\label{thm:cubic}
   Let the initial values $u(x,0)=v_0(x)-\bar v_0$ as described above.
   Assume \eqref{c:small_cubic}  
   and assume 
        $(\alpha_j)_{j\geq 0}$ is a positive decreasing sequence satisfying 
        \begin{equation}\label{cond:alphak;tkorder;f_cubic}
            \alpha_{j}\leq  \mu_{j}
	    \left(\frac{\alpha_{j-1}}{2}\right)^{1/\mu_{j}}  
	    \quad\text{for all $j\ge1$.}
        \end{equation}
Then:
  (i) The phase transition times $\tau=\tau(\Omega_j\setminus\tilde\Omega_j)$ satisfy $\tau_m=+\infty$ and $\tau_{j+1}>\tau_j$ for all $j\geq 0$, with 
        \begin{equation}\label{eq:tauk_ubd}
            e^{\tau_j}\leq \left(\frac{\alpha_j}{2}\right)^{-1/\mu_{j+1}} .
        \end{equation}
    (ii) If moreover for sufficiently large $j$  we have 
       \begin{equation} \label{c:alphak_cubic}
        \alpha_{j} \le \frac{1}{24}  \left( \frac{\alpha_{j-1}}{2}\right)^{1/\mu_{j}} e^{-2\kappa_{j}}  \,,
        \quad\text{where}\quad
            \kappa_{j} 
            := \frac{3}{\mu_{j+1}}\log\left(\frac{18}{\mu_{j+1}}\right) \,, 
        \end{equation}
        then $u(\cdot,t)$ does not converge as $t\to\infty$ (in any $L^p$, $1\le p<\infty$).
\end{theorem}
\begin{remark}\label{rmk:cubic_main}
    Note $\kappa_j\geq 1$ for all $j\ge0$, as $\mu_{j+1}\leq \mu_0 \leq \frac{1}{10}$.
    Condition \eqref{c:alphak_cubic} is much stronger than \eqref{cond:alphak;tkorder;f_cubic} as
    \[e^{-2\kappa_{j}}=\left(\frac{\mu_{j+1}}{18}\right)^{6/\mu_{j+1}}\ll\mu_{j}.
    \qedhere \]
\end{remark}

 \begin{figure}
    \centering
    \begin{tikzpicture}
    \begin{axis}[
        ymin = -0.1, ymax=0.8,
        ytick = {0,0.2,0.6},
        yticklabels={$\alpha_{j+2}$,$\alpha_j$,$\alpha_{j-2}$},
        xtick = \empty,
        axis lines = left,
        at={(0,-0.1)},
        x label style={at={(axis description cs:1,0)},anchor=north},
        y label style={at={(axis description cs:-0.007,1.07)},anchor=north, rotate=-90},
        xlabel = \(x\),
        ylabel = {\(v_0\)},
        legend pos=north west,
        width=0.9\textwidth,
        height=8cm,
    ]

    \addplot [
        domain=0:0.6, 
        samples=100, 
        color=black,
        style={thick},
    ] 
    {0};


    \addplot [
        domain=0.6:0.9, 
        samples=100, 
        color=black,
        style={thick},
    ] 
    {0.2*(1/(1+exp(-60*(x-0.75)))-0.5)+0.1};

    \addplot [
        domain=0.9:1.8, 
        samples=100, 
        color=black,
        style={thick},
    ] 
    {0.2};


    \addplot [
        domain=1.8:2.7, 
        samples=100, 
        color=black,
        style=thick,
    ] 
    {0.4*(1/(1+exp(-18*(x-2.25)))-0.5)+0.4};

    \addplot [
        domain=2.7:5, 
        samples=100, 
        color=black,
        style={thick},
    ] 
    {0.6};

    \end{axis}

    \draw[dotted](0,0) .. controls (0.7,-0.3) and (1.6,-0.3) .. (2.5,0);
    \node at (1.15,-0.5){$\Omega_{j+2}$};
    \node at (2.05,6.25){$\tilde\Omega_{j+2}$};
    
    \draw[dotted](2.5,0) .. controls (3.9,-0.3) and (5.6,-0.3) .. (7.15,0);    
    \node at (4.8,-0.5){$\Omega_{j}$};
    \node at (6,6.25){$\tilde\Omega_{j}$};

    \draw[dotted](7.15,0) .. controls (9.2,-0.3) and (11,-0.3) .. (13,0);
    \node at (10,-0.5){$\Omega_{j-2}$};

    \draw[dotted](1.6,0) -- (1.6,6);
    \draw[dotted](2.4,0) -- (2.4,6);
    \draw[dotted](1.6,6) .. controls (1.8,5.7) and (2.2,5.7) .. (2.4,6);


    \draw[dotted](7.15,0) -- (7.15,6);
    \draw[dotted](4.8,0) -- (4.8,6);
    \draw[dotted](4.8,6) .. controls (5.2,5.7) and (6.4,5.7) .. (7.15,6);


    \end{tikzpicture}
        \caption{Schematic illustration of smooth initial data near $\Omega_j$ for $j$ even, with transition zones $\tilde\Omega_j$ and $\tilde\Omega_{j+2}$.}

        \label{fig:smooth_data}
    \end{figure}
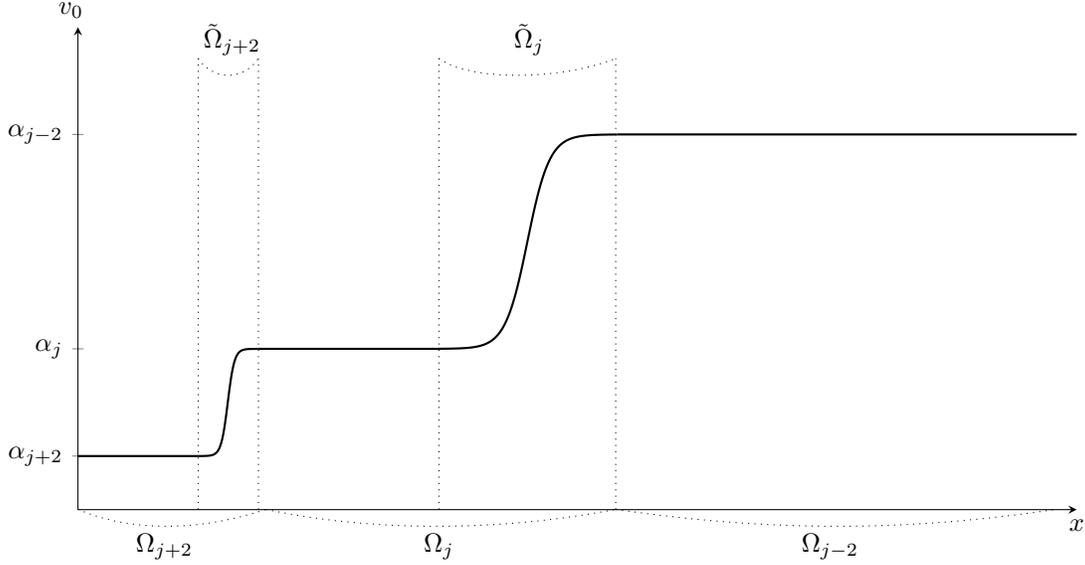

\begin{remark}[Smooth initial data]\label{rmk:smooth_init;cubic}
    To construct smooth initial data satisfying the assumptions of Theorem~\ref{thm:cubic}, first consider the case $\Omega=[0,1]$ with 
    intervals $\Omega_j$ of length $\mu_j$ defined as in Remark \ref{rmk:monotone_init;pwl}, but with the numbers $\frac14$ and $\frac34$ replaced by $\frac13$ and $\frac23$ respectively. 
    Fix a smooth, nondecreasing ``ramp'' function $\Theta:\R\to\R$ 
    such that $\Theta(x)=0$ for $x\le 1-\theta$ and $\Theta(x)=1$ for $x\ge1$. 
    Then set
    \begin{equation*}
    v_0(x) = \begin{cases}
 -1,\ 0,\ 1 &\text{ in}\quad \Omega_l,\ \Omega_m,\ \Omega_r \text{\ \ respectively,} \\
   \tilde v_j(x) &\text{ in } \Omega_j\,,\quad j=0,1,2,\ldots,
            \end{cases}
    \end{equation*}    
where $\tilde v_j$ is a 
    smooth function interpolating between $(-1)^j\alpha_j$ and $(-1)^j\alpha_{j-2}$ on $\Omega_j$, defined by 
    \[
    \tilde v_j(x) =
    \begin{cases}
   \phantom{+}\alpha_j + (\alpha_{j-2}-\alpha_j)\Theta\Bigl( \frac{x-c_j}{\mu_j}\Bigr)  
   & \text{for $j$ even, with } c_j=\inf \Omega_j,
   \\
   -\alpha_j - (\alpha_{j-2}-\alpha_j)\Theta\Bigl( \frac{c_j-x}{\mu_j}\Bigr)  
   & \text{for $j$ odd, with }  c_j=\sup \Omega_j.
   \end{cases}
    \]
   Then $\tilde v_j(x)=(-1)^j\alpha_j$ on $\Omega_j\setminus\tilde\Omega_j$, where
    $\tilde\Omega_j\subset\Omega_j$ is the closed interval of length $\theta\mu_j$ 
    at the right end of $\Omega_j$ for $j$ even (resp.\ at the left end for $j$ odd).
    Clearly $v_0$ is smooth everywhere in $[0,1]$ except possibly at the endpoints of $\Omega_m=[\frac13,\frac23]$. However, $v_0$ is smooth at these endpoints also, as a consequence of the fact
    that in $\tilde\Omega_j$ we have $\D_x^k v_0 = O(\mu_j^{-k}\alpha_{j-2})$ which approaches
    zero as $j\to\infty$ for each fixed $k\ge1$.

    We can make a similar construction of smooth initial data on a domain $\Omega=B(0,r)\subset\R^d$ for suitable $r>0$ using a radial construction. 
   Further, such radial initial data $u$ can be composed with any smooth volume-preserving diffeomorphism from $\R^d$ to $\R^d$ to produce more general smooth initial data in $\R^d$ with the same distribution of values.
\end{remark}

\subsection{Ordering of transition times}\label{ss:order_cubic}
We prove part (i) of Theorem~\ref{thm:cubic} in this subsection.
Henceforth, for $j=l,m,r$ we let $u_j(t)$ denote the value of $u(x,t)$ in $\Omega_j$. 
For $j=0,1,2,\ldots,$ we let $u_j(t)$ denote the value of $u(x,t)$ in 
$\Omega_j\setminus\tilde\Omega_j$, and we define $\tau_j=\tau(\Omega_j\setminus\tilde\Omega_j)$
be the corresponding phase transition time.
\begin{lemma}\label{lem:tau_m=infty}
The transition time $\tau_m=\tau(\Omega_m)=+\infty$.
That is, $u_m(t)\in\Phi_m$ for all $t\ge0$. 
\end{lemma}
\begin{proof}
By preservation of order we have 
\[
\nu(\{x: u(x,t)<u_m(t)\}) =\hat\mu_l = \tfrac13,
\qquad
\nu(\{x: u_m(t)<u(x,t)\}) =\hat\mu_r = \tfrac13.
\]
By invariance of $[a,b]$, if $u_m(t)$ escapes $\Phi_m$ on the right at some finite time $t_*$, 
then $u_m(t_*)=\hat a=-\frac12 a$ and 
\[
0 = \bar u > a \hat\mu_l + u_m(t_*)(\hat\mu_m+\hat\mu_r) = 0,
\]
a contradiction. Similarly, if $u_m(t_*)=\hat b = -\frac12 b$,
\[
0 = \bar u < u_m(t_*)(\hat\mu_l+\hat \mu_m)+ b \hat\mu_r = 0.
\]
Hence $u_m(t)\in\Phi_m$ for all $t\ge0$.
\end{proof}
Note now that by \eqref{e:tilde_order} and preservation of order, for any $x\in \Omega_j$, 
$u(x,t)$ can exit $\Phi_m$ only at $\hat b$ if $j$ is odd, and only at $\hat a$ if $j$ is even.
Since $(-1)^j(u(x,t)-u_j(t))\ge0$, the transition time 
\begin{equation}
    \tau(x)\le\tau_j \quad\text{ for any $x\in\tilde\Omega_j$, \quad $j=0,1,2,\ldots$}
\end{equation}
Let $\eps_l(t),\eps_r(t)$ be as defined in \eqref{d:ep_lr}.
Equivalently we have
\begin{align} 
    \eps_l(t) &= 
    \nu(\{x: \hat b\le u(x,0)<u_m(0) \text{ and } \tau(x)\ge t \})
\,,
\nonumber \\ 
    \eps_r(t) &= 
    \nu(\{x: \hat a\ge u(x,0)>u_m(0) \text{ and } \tau(x)\ge t \})
\,.
\label{d:ep_lr2}
\end{align}
These functions are left continuous in $t$. 
Because 
\[
\nu(\Omega_j\setminus\tilde\Omega_j) = (1-\theta )\mu_j 
\quad\text{ and }\quad
\sum_{k=0}^\infty \mu_{j+2k} = \frac{\mu_j}{1-\eta^2} ,
\]
by the assumption  $\theta \le\eta^2$ from \eqref{c:small_cubic} we have that
whenever $t\le \tau_j$ (so $u_j(t)\in[\hat b,\hat a]= \bar\Phi_m$),
\begin{equation}\label{i:eps_lr_trap}
(1-\eta^2)\mu_j \le 
(1-\theta )\mu_j
\le 
\left\{\!\begin{aligned}
&\eps_l(t) & \text{for $j$ odd}\\
&\eps_r(t) & \text{for $j$ even}
\end{aligned}
\right\}
\le \frac{\mu_j}{1-\eta^2}.
\end{equation}

        \begin{lemma}[Bounds on $u_m$]\label{lem:um_bd;tk;cubic}
Let $j_l(t)$ and $j_r(t)$ be defined as in Lemma~\ref{lem:umbounds}. Then 
            \begin{equation}\label{eq:um_bd;eps;general}
                \hat a - u_m(t)>\frac{(\hat b-a)\eps_l}{\hat\mu_m+\eps_r}>\mu_{j_l(t)},\qquad
                u_m(t)-\hat b>\frac{(b-\hat a)\eps_r}{\hat\mu_m+\eps_l}>\mu_{j_r(t)}.
            \end{equation}
        \end{lemma}

        \begin{proof}
        We will proceed as in the proof of Lemma \ref{lem:umbounds}. Preservation of order and invariance imply that
            \begin{align*}
            0=\bar u > a(\hat\mu_l-\eps_l) +\hat b\eps_l + u_m(t)(\hat\mu_m+\eps_r) + \hat a(\hat\mu_r-\eps_r) ,\\
            0= \bar u < \hat b(\hat\mu_l-\eps_l) + u_m(t)(\eps_l+ \hat\mu_m) + \hat a\eps_r + b (\hat\mu_r-\eps_r).
        \end{align*}
        The first inequality implies
        \begin{align*}
            \hat a - u_m(t)>
	    \frac{ \hat a(\hat\mu_m+ \eps_r)+ a(\hat \mu_l-\eps_l)+\hat b\eps_l+\hat a(\hat\mu_r-\eps_r)
	    }{\hat\mu_m+\eps_r} 
            =\frac{(\hat b-a)\eps_l}{\hat\mu_m+\eps_r},
        \end{align*}
where we used $a\hat\mu_l+\hat a\hat\mu_m+\hat a\hat\mu_r=0$.
Similarly, using the second inequality and
$\hat b\hat\mu_l+\hat b\hat\mu_m+ b\hat\mu_r=0$,
we obtain
        \begin{align*}
            u_m(t)-\hat b>\frac{
            -\hat b(\hat\mu_l-\eps_l)-\hat a\eps_r-b(\hat\mu_r-\eps_r)
            -\hat b(\hat\mu_m+\eps_l)}{\hat\mu_m+\eps_l}
            =\frac{(b-\hat a)\eps_r}{\hat\mu_m+\eps_l}.
        \end{align*}
Finally, the remaining bounds in \eqref{eq:um_bd;eps;general} follow by applying  
the first inequality in \eqref{i:eps_lr_trap} in the numerators 
 and the bounds  $\eps_l,\eps_r\le \mu_0/(1-\eta^2)<\frac1{9}$ in the denominators.
        \end{proof}

In order to obtain the proper ordering of transition times, we need to control
the expansion rate of $|u_m(t)-u_j(t)|$ inside the unstable phase $\Phi_m$.
For this purpose, note that
\begin{equation}\label{bd:f_ge}
|f(u)-f(v)|\ge  h|u-v|
\quad\text{ whenever }
\begin{cases}
u,v\in\Phi_m \text{ with } (\hat a-v)\wedge(v-\hat b)\ge h, \text{ or }\\
u,v\in\Phi_r \text{ with } v-\hat a\ge h, \text{ or }\\
u,v\in\Phi_l \text{ with } \hat b-v \ge h.
\end{cases}
\end{equation}
To see this, suppose $u,v\in\Phi_m$ and $h\leq \hat a-v\leq v-\hat b$. 
Then necessarily $h\le\hat a$, as $(\hat a-v)\wedge (v-\hat b)\leq \hat a$.
Explicitly computing, since $uv$ and $v^2$ are each less than $\hat a(\hat a-h)$, and 
$3\hat a^2=1$, we have
\begin{align*}
    \left|\frac{f(u)-f(v)}{u-v}\right|=1-u^2-uv-v^2
    \ge 1-\hat a^2 - 2\hat a(\hat a-h) = 2\hat a h \ge h.
\end{align*}
By symmetry, we can deduce the same inequality when $(v-\hat b)\leq (\hat a-v)$. Similarly, when $u,v\in\Phi_r$ and $v-\hat a\ge h$,
\[\frac{f(u)-f(v)}{u-v}=u^2+uv+v^2-1\geq (\hat a+h)^2+\hat a(\hat a+h)+\hat a^2 - 1 = \sqrt{3}h+h^2\geq h\]
and the case when $u,v\in\Phi_l$ can be verified by analogous calculations.

\begin{proof}[Proof of Theorem \ref{thm:cubic} part (i)]
By preservation of order we have $\tau_{j+2}>\tau_j$ for all $j\ge0$. 
Supposing that $\tau_{j+1}\le\tau_j$ for some $j$, we may take $j$ minimal. 
Then for $0\le t\le\tau_{j+1}\le\tau_j$, both $j_l(t),j_r(t)\le j+1$, 
so the bounds in Lemma~\ref{lem:um_bd;tk;cubic} apply to yield
\[
(\hat a-u_m)\wedge(u_m-\hat b)>\mu_{j+1}.
\]
In case $j$ is odd, we infer that for all $t\le \tau_{j+1}$, 
        \[
        \D_t(u_m-u_j) = f(u_j)-f(u_m) 
        \geq \mu_{j+1}(u_m-u_j), 
        \]
whence at $t=\tau_{j+1}$,
\begin{equation}\label{i:uj}
\hat a-\hat b > |u_m-u_j|\ge \alpha_j e^{\mu_{j+1}\tau_{j+1}}  \,.
\end{equation}
In case $j$ is even, the same inequality follows in similar fashion by computing $\D_t(u_j-u_m)$.
Now in either case, since $|f'|\le 1$ in $\Phi_m$ and $u_{j+1}(\tau_{j+1})=\hat a$ or $\hat b$,
use of Gronwall's inequality yields
\begin{equation}\label{i:uj+1}
\mu_{j+1}\le |u_m(\tau_{j+1})-u_{j+1}(\tau_{j+1})|\le \alpha_{j+1} e^{\tau_{j+1}} .
\end{equation}
The inequalities \eqref{i:uj}--\eqref{i:uj+1} imply
\[
\alpha_{j+1} \ge \mu_{j+1}e^{-\tau_{j+1}} > 
\mu_{j+1}\left(\frac{\alpha_j}{\hat a-\hat b}\right)^{1/\mu_{j+1}},
\]
which contradicts the assumption \eqref{cond:alphak;tkorder;f_cubic} since $\hat a-\hat b<2$. 
Hence $\tau_{j+1}>\tau_j$ for all $j$. 
The bound \eqref{eq:tauk_ubd} follows because in \eqref{i:uj} we can now replace $\tau_{j+1}$ by $\tau_j$.
\end{proof}

\subsection{Analysis of non-convergence}

By the result of part (i) of Theorem \ref{thm:cubic}, we have 
$\tau_{j+1}>\tau_j$ for all $j$. 
Then it follows that $|j_l(t)-j_r(t)|=1$ for all $t$, and whenever $t\leq \tau_j$, 
necessarily both $j_l(t),j_r(t)\le j+1$. Thus by Lemma \ref{lem:um_bd;tk;cubic},
 \begin{equation}\label{bd:um}
 \hat b+\mu_{j+1} < u_m(t) < \hat a-\mu_{j+1} \quad\text{whenever $t\le \tau_j$}.
 \end{equation}
 
In this section our goal is to prove part (ii) of Theorem~\ref{thm:cubic}.
The proof is more involved than in the piecewise-linear case. 
We proceed by examining the evolution of the phase ratio, 
then establish estimates involving exponential contraction in the stable phases,
and finish by an argument by contradiction.

\subsubsection{Evolution of the phase ratio}
Our strategy to obtain non-convergence is to study the evolution 
of the phase ratio $R$
defined exactly as in subsection~\ref{ss:heuristics_cubic}, by 
\begin{equation}\label{d:Ratio}
R = \frac{u_r-u_m}{u_m-u_l} \,.
\end{equation}
The evolution equation \eqref{eq:Revol1} 
continues to hold in the present context. 
In order to obtain an analog of \eqref{eq:Revol2},
we need to express the sum $u_l+u_m+u_r$ differently 
using conservation of mass.
For this purpose we alter the definition of $\eps_l(t),\eps_r(t)$
to always include whole pieces, as follows:

For any $t\ge0$, let 
$j_m(t)$ indicate the index of the {\em next} value $u_j$ to change phase
(by leaving $\bar\Phi_m=[\hat b,\hat a]$), so
\[
\jnex(t)=\min\{j:\tau_j\ge t\}=j_l(t)\wedge j_r(t).
\]
Then $t\in(\tau_{\jnex-1},\tau_{\jnex}]$. This means that if $j=j_m$ 
or $j_m+1$, 
then $u_j(t)\in \bar\Phi_m$ 
but $u_{j+2}(t)\notin\bar\Phi_m$, so $u(\tilde\Omega_j,t)$ may be split between phases.
For all {\em other} $j$, the sets $u(\Omega_j,t)$ are entirely in 
one phase---the unstable phase $\Phi_m$ if $t<\tau_j$, 
and one of the stable phases $\Phi_l$ or $\Phi_r$ if $t>\tau_j$, for $j$ odd or even
respectively.
Accounting only for those $j$ for which $u_j$ lies in (the closure of) the unstable phase, define
\begin{align}
\label{d:hatepsl}
\hat\eps_l(t) &= 
\nu\left( \bigcup \left\{ \Omega_j: \hat b\leq u_j(t)<u_m(t)\right\} \right)
= \sum_{{\rm odd}\,j\ge\jnex(t)}\mu_j
\,,\\
\hat\eps_r(t) &= 
\nu\left( \bigcup \left\{ \Omega_j: u_m(t)<u_j(t)\leq \hat a\right\} \right)
= \sum_{{\rm even}\,j\ge\jnex(t)}\mu_j
\,.
\label{d:hatepsr}
\end{align}
\begin{remark} 
The relation with $\eps_l,\eps_r$ is as follows. Let us denote the part of $\Omega_j$ outside $\bar\Phi_m$ by
\[
\nu_{j}(t) = \nu(\{x\in\Omega_{j}:u(x,t)<\hat b \text{ or } u(x,t)>\hat a\}),
\]
and note 
$0\le \nu_{j}(t)\le\theta \mu_{j}$ for $j\geq \jnex(t)$ because 
$u_{j}\in[\hat b,\hat a]$.  Then 
\begin{equation}
\begin{array}{lll}
j=\jnex(t) \text{ odd}\implies 
& \eps_l = \hat\eps_l-\nu_{\jnex}\,, &\eps_r=\hat\eps_r-\nu_{\jnex+1}\,,\\
j=\jnex(t) \text{ even}\implies 
&\eps_r = \hat\eps_r - \nu_{\jnex} \,, 
& \eps_l = \hat\eps_l -\nu_{\jnex+1} \,.
\end{array}
\end{equation}
\end{remark}
Recall that 
\begin{equation}
\bar u  = 
\mu_l u_l + \mu_m u_m + \mu_r u_r + 
\sum_{j\ge0} 
\left( (1-\theta)\mu_j u_j + \int_{\tilde\Omega_j}u \right) \,.
\end{equation}
Since $\nu(\tilde\Omega_j)=\theta\mu_j$, for any constant $v\in\R$ we can write  
\[ (1-\theta)\mu_j u_j + \int_{\tilde\Omega_j}u 
=  (1-\theta)\mu_j (u_j-v) 
+ \int_{\tilde\Omega_j} (u-v)\,d\nu 
 +\mu_j v   \,.
\]
In view of \eqref{e:mu_lmr_cubic} and \eqref{d:hatepsl}--\eqref{d:hatepsr} 
then, we find that
\begin{equation}
\bar u = 
\ \  u_l(\hat\mu_l - \hat\eps_l) 
+ u_m(\hat\mu_m + \hat\eps_l+\hat\eps_r) 
+ u_r(\hat\mu_r - \hat\eps_r) 
+ H(t)  \,,
\end{equation}
where $H(t) = H_l(t) +H_m(t) +H_r(t)$, with %
\begin{align}
H_l(t) &=  \sum_{{\rm odd}\, j< \jnex(t)} 
\left(
(1-\theta)\mu_j (u_j-u_l) + \int_{\tilde\Omega_j} (u(x,t)-u_l)\,d\nu(x)  
\right)\,,
\label{d:Hl}\\ 
H_m(t) &=  \sum_{j\ge\jnex(t)} 
\left(
(1-\theta)\mu_j (u_j-u_m) + \int_{\tilde\Omega_j} (u(x,t)-u_m)\,d\nu(x)  
\right)\,,
\label{d:Hm}\\ 
H_r(t) &=  \sum_{{\rm even}\, j< \jnex(t)} 
\left(
(1-\theta)\mu_j (u_j-u_r) + \int_{\tilde\Omega_j} (u(x,t)-u_r)\,d\nu(x)  
\right)\,.
\label{d:Hr}
\end{align}
Since $\bar u=0$ and by \eqref{e:hatmu_cubic} we obtain our desired relation, 
\begin{equation}\label{eq:ulumur}
\frac13(u_l+u_m+u_r) =
\hat\eps_r(u_r-u_m)-\hat\eps_l(u_m-u_l)
-H(t)\,.
\end{equation}

Now, by using \eqref{eq:ulumur} in the evolution equation
\eqref{eq:Revol1} for the phase ratio $R$, we infer that
\begin{equation}
\D_t R =   
3\Bigl(\hat\eps_l(u_m-u_l)-\hat\eps_r (u_r-u_m)+ H(t)\Bigr)
    (u_r-u_l) R  \,,
\label{eq:Revol3}
\end{equation}
which can be compared to equation \eqref{eq:Revol2} for solutions with
three values.  As this comparison suggests, our aim is show that
$ H(t)$ is tiny enough over large enough time intervals 
that non-convergence follows.

\subsubsection{Estimates in the stable phases}

\begin{lemma}[Estimates on $\bar f,u_l,u_r$]\label{lem:bd_ulr}
Let $h_j = \frac13\mu_{j+1}$ for all $j$.
Then whenever $t\le \tau_j$ we have 
\begin{equation}
u_l(t)<\hat b-h_j,
\qquad   u_r(t)>\hat a+h_j,
\qquad f(\hat a+ h_j )<\bar f(t)< f(\hat b- h_j).
\end{equation}
\end{lemma}
\begin{proof}
We prove the bounds on $\bar f(t)$ first. Observe
\begin{align*}
\bar f(t) &> f(a)(\hat\mu_l-\eps_l) + f(u_m)(\hat\mu_l+\eps_l) + f(\hat a)\hat\mu_r
\\ &= f(\hat a) + (f(u_m)-f(\hat a))(\hat\mu_l+\eps_l)
\\ &> f(\hat a) + (f(\hat a - \mu_{j+1}) - f(\hat a))(\hat\mu_r + \mu_{j+1}) \,,
\end{align*}
\begin{align*}
\bar f(t) &< f(\hat b)\hat\mu_l + f(u_m)(\hat\mu_m+\eps_r) + f(b)(\hat\mu_r-\eps_r)
\\ &= f(\hat b) + (f(u_m)-f(\hat b)) (\hat\mu_m+\eps_r)
\\ &\le f(\hat b) + (f(\hat b+\mu_{j+1})-f(\hat b))(\hat\mu_m +\mu_{j+1}) \,.
\end{align*}
Since $0=f'(\hat b)=f'(\hat a)$, Taylor expansion gives, for $0<h<1/\sqrt3 =\hat a=-\hat b$,
\begin{align*}
f(\hat a - h) - f(\hat a) &= 3\hat a h^2 - h^3 >  2\hat a h^2 >  h^2,
\\ 
f(\hat b+h) - f(\hat b) &= 3 \hat b h^2+ h^3 <  - 2\hat a h^2 < -h^2, 
\end{align*}
Hence 
\begin{align}
\bar f(t) &> f(\hat a) + \mu_{j+1}^2 (\tfrac13+\mu_{j+1}) 
> f(\hat a + \tfrac13\mu_{j+1})\,,
\\
\bar f(t) &< f(\hat b) - \mu_{j+1}^2 (\tfrac13+\mu_{j+1})
< f(\hat b - \tfrac13\mu_{j+1}) \,.
\end{align}
This proves the claimed bounds on $\bar f$.

Note that  initially $u_l(0)=-1-\bar v_0<\hat b-h_0$ and $u_r(0)=1-\hat v_0>\hat a+h_0$, 
since $|\bar v_0|\le\mu_0\le\frac1{10}$ and $h_0\le \frac1{240}$.
Then the claimed bounds on  $u_l$ and $u_r$ follow from the bounds on $\bar f$,
the evolution equation \eqref{eq:ODE_nonlocal},
and the monotonicity of $f$ on the invariant intervals $[a,\hat b]$ and $[\hat a,b]$.
\end{proof}

\begin{lemma}\label{lem:bd_urjl}
    Whenever $\tau_j<t<\tau_{k}$ we have:
    \begin{align*}
     u_r-u_j\le  e^{-h_k(t-\tau_j)} \quad\text{for $j,k$ even} ,
\qquad     
    u_j-u_l\le  e^{-h_k(t-\tau_j)} \quad\text{for $j,k$ odd}.
    \end{align*}
\end{lemma}
\begin{proof} 
Suppose $\tau_j<t<\tau_{k}$. In case $j,k$ are both even,
we know $u_r\ge \hat a+ h_k$ by Lemma~\ref{lem:bd_ulr}, so
\[
\frac{f(u_r)-f(u_j)}{u_r-u_j} \ge  \frac{f(\hat a+h_k)-f(\hat a)}{h_k}
= 3\hat a h_k+h_k^2 > h_k \,.
\]
It follows
\begin{equation*}
\D_t(u_r-u_j) = -(u_r-u_j)
\frac{f(u_r)-f(u_j)}{u_r-u_j} < -h_k (u_r-u_j) \,,
\end{equation*}
hence
\begin{equation}
u_r(t)-u_j(t) \le 
(u_r(\tau_j)-u_j(\tau_j)) e^{-h_k(t-\tau_j)}
\le (b-\hat a) e^{-h_k(t-\tau_j)}
<  e^{-h_k(t-\tau_j)}\,.
\end{equation}

In case $j,k$ are odd, we know $u_l\le \hat b- h_k$, hence
\[
\frac{f(u_j)-f(u_l)}{u_j-u_l} \ge  
\frac{f(\hat b)-f(\hat b-h_k)}{h_k}
= -3\hat b h_k+h_k^2 > h_k.
\]
It follows
\begin{equation*}
\D_t(u_j-u_l) = -(u_j-u_l)
\frac{f(u_j)-f(u_l)}{u_j-u_l} < -h_k (u_j-u_l) \,,
\end{equation*}
thus
\begin{equation}
u_j(t)-u_l(t) \le 
(u_j(\tau_j)-u_l(\tau_j)) e^{-h_k(t-\tau_j)}
\le (\hat b- a) e^{-h_k(t-\tau_j)}
<  e^{-h_k(t-\tau_j)} \,.
\end{equation}
\end{proof}

\begin{lemma}[Bounds for $ H(t)$]\label{lem:Hbd}
    For $0\le \tau_{k-1}<t<\tau_k$  we have 
    \begin{equation*}
        |H_m(t)| \le 2\mu_k (\alpha_k e^t + \theta) \,, 
        \qquad
        |H_l(t)+H_r(t)| \le  2\mu_0 e^{-h_k(t-\tau_{k-1})} \,.   
  \end{equation*}
\end{lemma}
\begin{proof}
    For the given range of $t$ we have  $j_m(t)=k$. 
    To prove the bound on $H_m(t)$ defined by \eqref{d:Hm}, 
    we use the bound $|f'|\le 1$ in $\Phi_m$ to infer $|u_m-u_j|\le \alpha_je^t$
    for all $j\ge k$, and the bound
    $|u-u_m(t)|<b-\hat b=\sqrt3$ in $\tilde\Omega_j$. 
    Then since $\eta\le\frac18$ we infer
    \begin{equation}
        |H_m(t)|\le \sum_{j\ge k}
        \bigl( (1-\theta)\mu_j \alpha_j e^t + \theta\mu_j(b-\hat b) \bigr) 
        \le  
        \frac{\mu_k}{1-\eta}  (\alpha_ke^t + \sqrt3 \theta) 
        \le 2\mu_k( \alpha_k e^t+\theta)\,.
    \end{equation}
    By Lemma~\ref{lem:bd_urjl} we find
    \begin{equation*}
        |H_l(t)+H_r(t)| \le 
        \sum_{j<k} \mu_j e^{-h_k(t-\tau_j)} \le 2\mu_0 e^{-h_k(t-\tau_{k-1})} \,.
        \qedhere
    \end{equation*}
\end{proof}

Note that  $H_m(t)$ can be kept small for any specified time by forcing the $\alpha_k$ to 
decay faster, whereas  the exponential contraction in the stable phase will force 
$H_l(t),H_r(t)$ to be small for $t-\tau_{k-1}$ large enough.
We will see that smallness of $H(t)$ implies lower bounds on the drift of the phase ratio 
$R$ in \eqref{eq:Revol3}, leading to non-convergence.

\subsubsection{Proof of non-convergence}

In this subsection we complete the proof of part (ii) of Theorem~\ref{thm:cubic}. 
For use below, recall $h_k=\frac13\mu_{k+1}$, and
note that $\kappa_k$ satisfies
\begin{equation}\label{bds:kappa}
     \kappa_{k}=\frac{1}{h_k}\log\left(\frac{6}h_k\right) ,\qquad
h_k e^{\mu_k\kappa_k} = 6\,,
\qquad
  2\mu_0 e^{-h_k\kappa_k}   \le \frac19\mu_{k+1}  \,.
\end{equation}

\begin{proof}[Proof of Theorem \ref{thm:cubic} part (ii)]
1. We argue by contradiction.
Supposing that $\lim_{t\to\infty} u(\cdot,t)$ exists, there is some 
   $\hat s\in[f(a),f(b)]$ such that  as $t\to\infty$, 
   \[
   \bar f(t)\to \hat s \quad\text{and} \quad u_j(t)\to z_j(\hat s) \quad\text{for $j=l,m,r$}.
  \]
We will consider the cases $\hat s\le0$ and $\hat s>0$ separately. 
First consider the case $\hat s\le 0$. Then 
necessarily $z_l(\hat s)\le -1$  and $0\le z_m(\hat s)\le z_r(\hat s)\le1$, 
and as $t\to\infty$ we have 
\begin{equation}\label{e:Rlim_le0}
R(t)\to \hat R:= \frac{z_r(\hat s)-z_m(\hat s)}{z_m(\hat s)-z_l(\hat s)} \in[0,1].
\end{equation}
In particular, if $T$ is large enough, then for all $t>T$ we have 
\begin{equation}\label{bd:R<2}
u_l(t) \le -1+\tfrac12\eta^2, \quad 
u_m(t) \ge -\tfrac12\eta^2, \quad\text{and}\quad
R(t) < 2\,.
 \end{equation}
We will contradict the last conclusion by showing that for any sufficiently large odd $k$, 
necessarily $R(t_k)\geq 12$ for some $t_k\in (\tau_{k-1},\tau_k)$.

2. We claim that for any sufficiently large odd $k$ with $\tau_{k-1}>T$, 
\begin{equation}\label{bd:tauk}
\tau_k>\tau_{k-1}+2\kappa_k \,.
\end{equation}
Indeed, since $k$ is odd and $\eta\leq \frac18$, 
\[
\alpha_ke^{\tau_k}= u_m(\tau_k)-u_k(\tau_k)= u_m(\tau_k)-\hat b >
-\frac12\eta^2 +\frac1{\sqrt 3} > \frac14 \,.
\]
But in light of \eqref{eq:tauk_ubd} and the condition \eqref{c:alphak_cubic}, 
we get that for all $t\le\tau_{k-1}+2\kappa_k$, 
\begin{equation} \label{bd:alphaket}
\alpha_k e^t 
\le 
\alpha_k   \left( \frac{2}{\alpha_{k-1}}\right)^{1/\mu_k} e^{2\kappa_k}
\le \frac{1}{24}  \,.
\end{equation}
Thus \eqref{bd:tauk} holds.

3. For $k$ odd and $T<\tau_{k-1}<t<\tau_k$, we have
\[ 
j_m(t)=k, \qquad \hat\eps_l(t) = \frac{\mu_k}{1-\eta^2}, \qquad\hat \eps_r(t) = \eta\hat\eps_l,
\]
and equation \eqref{eq:Revol3} takes the form
\begin{equation}\label{eq:dtR_kodd}
\D_t R = 3\mu_k \left( (1-\eta R) \frac{u_m-u_l}{1-\eta^2}
+ \frac{H(t)}{\mu_k} \right)(u_r-u_l)R \,.
\end{equation}
Now we can deduce from Lemma~\ref{lem:Hbd}, \eqref{c:small_cubic}, \eqref{bd:alphaket} and \eqref{bds:kappa} that for $\tau_{k-1}+\kappa_k< t<\tau_{k-1}+2\kappa_k$, 
\begin{equation}\label{bd:H_mlr}
\frac{|H_m(t)|}{\mu_k} \le  2(\alpha_ke^t + \theta)
\le \frac16 \,, \qquad
\frac{|H_l(t)+H_r(t)|}{\mu_k} \le 
\frac{2\mu_0 e^{-h_k\kappa_k}}{\mu_k} \le \frac19 \,.
\end{equation}
It follows from~\eqref{bd:R<2},
\eqref{eq:dtR_kodd}, the fact $u_r-u_l>\hat a -\hat b>1$ 
and $\eta R \le\frac14$ that for all $t$ in this range,
\begin{equation}
\D_t R \ge \mu_k R \,. 
\end{equation}
Using \eqref{bd:um} and Lemma~\ref{lem:bd_ulr} we can ensure
$u_r-u_m \ge 4 h_k$, hence for $t=\tau_{k-1}+\kappa_k$, 
\[
R(\tau_{k-1}+\kappa_k)= \frac{u_r-u_m}{u_m-u_l}
\ge \frac{4h_k}{\hat a -a}\ge 2h_k \,.
\]
Using \eqref{bds:kappa}, we infer that at time $t_k:=\tau_{k-1}+2\kappa_k$,
\begin{equation}
    R(t_k) \ge 2h_k e^{\mu_k\kappa_k} \ge 12\,.
\end{equation}
This  contradicts $R(t)<2$ for all $t>T$, and concludes the analysis in the case $\hat s\le0$.

4. The treatment in the case $\hat s>0$ is broadly similar. 
In this case, we can say that 
\begin{equation}\label{e:Rinvlim_le0}
R(t)\inv\to \check R:= \frac{z_m(\hat s)-z_l(\hat s)}{z_r(\hat s)-z_m(\hat s)} \in[0,1],
\end{equation}
and find $T$ large enough so that for all $t>T$,
\begin{equation}
    u_r(t)> 1-\tfrac12\eta^2, \quad u_m(t) \le \tfrac12\eta^2, \quad\text{and}\quad
    R(t)\inv <2.
\end{equation}
Now taking $k$ even and sufficiently large, such that $\tau_{k-1}>T$,
\[
\alpha_ke^\tau_k = u_k(\tau_k)-u_m(\tau_k) > \hat a -\frac12\eta^2 > \frac14 ,
\]
while \eqref{bd:alphaket}, and hence \eqref{bd:tauk}, follow as before.
For $k$ even and $T<\tau_{k-1}<t<\tau_k$, 
\[ 
j_m(t)=k, 
\qquad\hat \eps_l(t) = \eta\hat\eps_r,
\qquad \hat\eps_r(t) = \frac{\mu_k}{1-\eta^2}, 
\]
and we find \eqref{eq:Revol3} equivalent to
\begin{equation}
\D_t R\inv = 
3\mu_k \left( (1-\eta R\inv) 
\frac{u_r-u_m}{1-\eta^2}
- \frac{H(t)}{\mu_k} \right)(u_r-u_l) R\inv
\end{equation}
As before,  for $\tau_{k-1}+\kappa_k< t<\tau_{k-1}+2\kappa_k$, 
the bounds \eqref{bd:H_mlr} hold, and we can infer
\begin{equation}
    \D_t R\inv \ge \mu_k R\inv
\end{equation}
for all $t$ in this interval. At the time $t=\tau_{k-1}+\kappa_k$ we have
\begin{equation}
R(\tau_{k-1}+\kappa_k)\inv= \frac{u_m-u_l}{u_r-u_m}
\ge \frac{4h_k}{b - \hat b}\ge 2h_k \,,
\end{equation}
and infer $R(\tau_{k-1}+2\kappa_k)\inv\ge 12$ like before, obtaining a contradiction.
This finishes the proof.
\end{proof}

    \subsection{Unstable nature of non-convergence}\label{ssec:unstable_counterex}
    
    Now we present a proof of Proposition \ref{prop:unstable_counterex}.
    
    \begin{proof}[Proof of Proposition \ref{prop:unstable_counterex}]
    1. Suppose $u$ does not converge in $L^2(\Omega)$ as $t\to\infty$. As $f=u^3-u$ is not constant on any open interval, 
we deduce $\bar f(t)$ does not converge either; see \cite[Lemma 3.4]{Novick-CohenPego91}. Hence we may choose
an open interval $J$  such that 
\begin{equation}\label{d:J}
\liminf \bar f(t) < \inf J < \sup J <  \limsup \bar f(t) \,,
\end{equation}
and such that $\bar J$ contains only regular values of $f$, omitting both critical values $f(a)$ and $f(b)$.  
Moreover, due to the fact from \eqref{limdt2}
that $\|\D_tu\|_{L^2}\to0$ as $t\to\infty$, we infer by differentiating \eqref{d:fbar} that the Lipschitz function $\bar f(t)$ has derivative
$\D_t\bar f(t)\to0$ as $t\to\infty$ in its set of differentiability. 
If we let $J_m$ denote the ``middle third'' of $J$, then it follows there exist sequences $T_k\to\infty$ and $\tau_k\to\infty$
such that 
\begin{equation}
    \bar f(t)\in J_m \quad\text{ for all $t\in I_k:=[T_k,T_k+\tau_k]$ and all $k\in\N$.}
\end{equation}
Letting $\delta=|J_m|$ denote the length of $J_m$, we have $|s-\hat s|\ge\delta$
whenever $s\in J_m$ and $\hat s\notin J$.

2. We first dispose of the possibility that  $\sup J<f(a)$ or $f(b)<\inf J$.
In this case $f$  has a unique local inverse $z(s)$ defined for $s\in J$ satisfying $f(z(s))=s$,
and $f$ is strictly monotone increasing on $z(J)$. 
Thus the interval $z(J)$ is pointwise stable during each interval $I_k$ 
(cf.~Lemma~\ref{lem:invar_rgn}), for when $\bar f(t)\in J_m$, 
we have 
\begin{equation}\label{e:ddrift}
\text{$-f(v)+\bar f(t)\le -\delta<0$ \quad if $v\ge\sup z(J)$, 
\quad and \quad $-f(v)+\bar f(t)\ge \delta>0$ \quad if $v\le \inf z(J)$.}
\end{equation}
Moreover, whenever $k$ is so large that $\tau_k=|I_k|\ge \hat\tau:= 2M/\delta$ where $M:=\sup|u|+1$,
then 
\begin{equation}\label{e:inzj}
u(x,T_k+\hat\tau)\in z(J) \quad\text{ for all $x\in\Omega$. }
\end{equation}
The reason is that  $|u(x,T_k)|\le M$,
and the quantity $u(x,t)$, if not initially in $z(J)$, 
must monotonically move toward it with speed exceeding $\delta$, by \eqref{e:ddrift}. 
Since $\hat\tau\delta=2M$,
$u(x,t)$ must enter $z(J)$ before time $T_k+\hat\tau$, and cannot escape as long as $t\in I_k$.

But now, since \eqref{e:inzj} holds, Lemma~\ref{lem:invar_rgn} implies the interval $z(J)$
becomes positively invariant and therefore $u(x,t)\in z(J)$ for all large $t$. 
This forces $\bar f(t)\in J$ ever after, contradicting the choice of $J$ in \eqref{d:J}.
By consequence we must have 
\[
J\subset \hat J = (f(a),f(b)). 
\]
In particular, $a<f\inv(J)<b$, i.e., $a<v<b$ whenever $f(v)\in J$.

3.  By the invariance arguments of Lemma~\ref{lem:invar_rgn}, 
 the phase intervals $[-M,\hat b]$ and $[\hat a,M]$ are pointwise stable
 during the intervals $I_k$ when $\bar f(t)\in J_m$.  Supposing $k$ is so large that $\tau_k>\hat\tau$,
 for a similar reason as in step 2 it follows that 
 if $u(x,T_k)\in [-M,\hat b]$ then $u(x,T_k+\hat\tau)\in z_l(J)$,
 and if  $u(x,T_k)\in [\hat a,M]$ then $u(x,T_k+\hat\tau)\in z_r(J)$.
 In particular this implies that there exists some $T_*=T_k+\hat\tau$ such that $u(x,T_*)\in[a,b]$ for all $x$.

For $t>T_*$, $[a,b]$ is positively invariant and the phase intervals $\Phi_l$ and $\Phi_r$ are pointwise stable.

Then the sets defined for $t>T_*$ by 
\[
  \Omega_j(t)=\{x:\,u(x,t)\in \Phi_j\} \quad\text{ for } j=l,m,r,
\]
are monotonic for $t>T_*$. Indeed, the set $\Omega_m(t)$ decreases in time whereas the sets $\Omega_l(t)$ and $\Omega_r(t)$ increase in time. 
Thus, for each $j=l,m,r$, the quantities
\[
  \Omega_j^\infty=\lim_{t\uparrow \infty}\Omega_j(t) \quad\text{ and }\quad
   \mu_j^\infty=\nu(\Omega_j^\infty)
\]
exist.  Let $\check\Omega(t)=\Omega_m(t)\setminus\Omega_m^\infty$ denote the ``bad set'' where
$u(x,t)$ is not in the phase it eventually enters.

4.  We next claim that 
\begin{equation}\label{e:zsum}
\sum_j\mu_j^\infty z_j(s)=\bar u
\quad\text{for all $s\in J$.} 
\end{equation}
Fix $s\in J$, and define $\phi(x)=z_j(s)$ for $x\in\Omega_j^\infty$, \blue$j=l,m,r$,
so that\nc $f(\phi(x))=s$ for all $x$ and $\bar\phi = \sum_j \mu_j^\infty z_j(s)$.
Note that for some $\beta>0$ we have 
\[
|v-z_j(s)|\le \beta|f(v)-s| \quad\text{for all $v\in\Phi_j$, $j=l,m,r$,}
\]
hence for $t>T_*$,
\[
|u(x,t)-\phi(x)|\le \beta |f(u(x,t))-s| \quad\text{for all $x\notin \check\Omega(t)$.}
\]
Taking $t$ along any sequence $t_k\to\infty$ such that $\bar f(t_k)=s$ and $t_k>T_*$, we deduce
that 
\begin{align*}
|\bar u-\bar\phi|^2\le
\int_\Omega|u(x,t)-\hat u(x)|^2\,d\nu\le 
\beta^2\int_{\Omega\setminus\check\Omega(t)} |f(u(x,t))-\bar f(t_k)|^2\,d\nu
+ 2M \nu(\check\Omega(t)) \to 0
\end{align*}
as $k\to\infty$, since $ \int_\Omega|f(u)-\bar f|^2\,d\nu = \|\D_tu\|^2_{L^2} \to0$.
Hence $\bar u=\bar\phi$, and this proves \eqref{e:zsum}.

Property \eqref{e:zsum} implies that $\mu_j^\infty=\frac13$ for each $j=l,m,r$,
by Proposition 12 of \cite{BallSengul15}, which concerns relations between roots of cubic-like
analytic functions. Then it follows $\bar u=0$,
since the $z_j$ are the three roots of the cubic $u^3-u-s$.

5. For the remainder of the proof, fix some $\hat x\in \Omega_m^\infty$
and let $c=u(\hat x,0)$. We claim that 
\begin{equation}
\Omega^\infty_m = \hat E  \qquad\text{where}\quad
\hat E:= \{x\in\Omega: u(x,0)=c \}  .
\end{equation}
Here $\hat E$ denotes the level set where $u(x,0)=c=u(\hat x,0)$.
We can then infer that $\nu(\hat E)=\mu_m^\infty=\frac13$,
and this will almost finish the proof. 

By \eqref{eq:ODE_nonlocal}, $u(x,t)=u(\hat x,t)$ for all $x\in \hat E$ and all $t\ge0$, so $\hat E\subset\Omega_m^\infty$.
Suppose then that some $x$ exists in $\Omega_m^\infty\setminus \hat E$.
Then $u(x,0)\ne u(\hat x,0)$, yet both $u(x,t)$ and $u(\hat x,t)$ lie in $\Phi_m$ for all $t\ge T_*$. 
It remains to show this leads to a contradiction.

Because $f$ is decreasing on $\Phi_m$ and 
\[
\D_t (u(x,t)-u(\hat x,t)) = -f(u(x,t))+f(u(\hat x,t)) \,,
\]
the difference $h(t)=|u(x,t)-u(\hat x,t)|$ is increasing for all $t>T_*$.
Moreover, $\D_th(t)\ge\eta$ for some $\eta>0$, such that 
$f(v)-f(w)\ge\eta$ whenever $v,w\in\bar\Phi_m$ with $v+h(T_*)<w$.
This forces $h(t)>\hat a-\hat b$ after time $T_*+ (\hat a-\hat b)/\eta$,
which contradicts that both $u(x,t)$ and $u(\hat x,t)$ lie in $\Phi_m=(\hat b,\hat a)$.

Hence $\Omega_m^\infty=\hat E$. For each point 
$x\in E_+:=\{x\in\Omega:u(x,0)>c\}$,
preservation of order and the argument just made imply that $u(x,t)\in\Phi_r$
for $t$ large enough, and for each point 
$x\in E_-:=\{x\in\Omega:u(x,0)<c\}$, necessarily $u(x,t)\in\Phi_l$ for $t$ large enough.
Then it follows $E_+=\Omega^\infty_r$ and $E_-=\Omega^\infty_l$, whence
$\nu(E_+)=\mu^\infty_r=\frac13$ and $\nu(E_-)=\mu^\infty_l=\frac13$.
This completes the proof.
    \end{proof}

\section{Sensitivity of convergence rates}\label{sec:sensitivity}
  In this section we comment on the possibility of curiously high sensitivity of convergence rates of solutions of the 
  finite-dimensional system \eqref{eq:uRN} to perturbations of parameters involving degenerate equilibria. 
{\purp  This connects with the gradient inequality in Lemma~\ref{lem:gradE}(ii), which holds under the hypothesis
  that the state $\hat\uu\in\R^N$ lies on a curve $\{\phib(s)\}_{s\in\hat J}$ of regular equilibria
  with constant average $\ip{\one,\phib(s)}$.  
  In this situation, provided we happen to know that $\hat \uu$ lies in the $\omega$-limit set 
  of some solution $\uu(t)$ of \eqref{eq:uRN}, the gradient inequality implies, by a simple and classical calculation,
  that $\uu(t)$ converges to $\hat\uu$ as $t\to\infty$ at an exponential rate. 

  A small perturbation of parameters can drastically alter the asymptotic rate of convergence, however,
  even if the asymptotic limit is not changed much. 
\nc}  Consider the three-value case for piecewise-linear $f$, recalling from Section \ref{ss:heuristic_PL} that, upon fixing $\nu_j=\hat\mu_j-\eps_j$ for $j=l,m,r$ with $\eps_m=-\eps_l-\eps_r$, \eqref{eq:dtbarf} implies
  \begin{align*}
      \frac{d}{dt}\bar f(t)=-2(\eps_r+\eps_l)\bar f(t) + (\eps_l-\eps_r-\bar u) \,.
  \end{align*}
  When $\eps_r+\eps_l>0$ and  $-2\eps_r< \bar u <2\eps_l$,
  \begin{align*}
      \frac{d}{dt}\bar f(t)=-2(\eps_r+\eps_l)\left(\bar f(t)- \barfeq 
      \right) \,,
      \qquad \barfeq =  \frac{\eps_l-\eps_r-\bar u}{2(\eps_r+\eps_l)} \,,
  \end{align*}
  and $\bar f(t)$ contracts towards $\barfeq \in (-\frac12,\frac12)$ at an $O(\eps_r+\eps_l)$-exponential rate.
  
  On the other hand, setting $\eps_l=\eps_r=0$, we see
  \[\bar f(t)=\bar f(0) \text{ for all } t\geq 0.\]
  Thus, when $u_m(0)=\bar f(0)$, $u_m$ is stationary and we observe an $O(1)$ exponential convergence rate of the solution, as
  \[\partial_t u_j(t)=-(u_j(t)-\bar f(0)) \quad\text{ for } j=l,r.\]
{\purp  In fact, this can also be seen via the gradient inequality in Lemma \ref{lem:gradE}; 
  letting $\phi_j(s) =-1+s,-s,1+s$ for $j=l,m,r$  for $s\in(-\frac12,\frac12)$, 
  we see $\sum_j \nu_j\phi_j(s) \equiv 0 =\bar u$.
   Thus the gradient inequality in Lemma \ref{lem:gradE}(ii) holds and becomes a \Loja\ inequality with $O(1)$ constant, implying exponential convergence at a rate that is $O(1)$.
   \nc}

    Similar sensitivity can be observed for the cubic nonlinearity. Considering again the three-valued case and setting $R=\frac{u_r-u_m}{u_m-u_l}$, recall from \eqref{eq:Revol1} that
    \[\partial_t R=-(u_l+u_r+u_m)(u_r-u_l)R.\]
    As
    \[\bar u = \sum_{j=l,m,r}\mu_j u_j = \frac13 (u_l+u_r+u_m)+\eps_l(u_m-u_l)-\eps_r(u_r-u_m),\]
    we have
    \begin{align*}
    \partial_t R&=-3(\eps_r(u_r-u_m)-\eps_l(u_m-u_l)+\bar u)(u_r-u_l)R,\\
    \partial_t R^{-1}&=-3(\eps_l(u_m-u_l)-\eps_r(u_r-u_m)-\bar u)(u_r-u_l)R^{-1}.
    \end{align*}
    For $\bar u = 0$ and for small $\eps_r,\eps_l>0$, the ratio $R$ evolves toward the equilibrium $\eps_l/\eps_r$ 
    at a slow exponential rate that is $O(\eps_r)$.
    And for $\eps_r=\eps_l=0$, when $\bar u>0$ (resp.~$-\bar u>0$), the ratio $R$ (resp. $R^{-1}$) contracts exponentially toward zero at a rate that is $O(\bar u)$.
    
    In case $\bar u=\eps_l=\eps_r=0$, however, the ratio $R$ is invariant in time, and $O(1)$ exponential convergence can be observed. For instance, if
    \[\bar f(0)=u_m(0), \qquad u_l(0)=-u_r(0),\]
    we see $u_m$ remains constant at $0$ and $u_l=-u_r$, as $R(t)=R(0)=1$ implies $3\bar u = u_r+u_l+u_m= 3u_m =0$. As $f$ is symmetric about $0$, this means $\bar f(t) = 0$ for all $t\geq 0$, and thus
    \[\partial_t u_j = -u_j(u_j+1)(u_j-1) \quad\text{ for } j=l,r.\]
    Then $u_r$ and $u_l$ converge exponentially towards $1$ and $-1$ respectively with $O(1)$ rate.

    In summary, even for finite-dimensional dynamics where convergence to equilibrium always occurs, the exponential {\em rate} of convergence
    for the gradient system \eqref{eq:ODE_nonlocal} can suddenly jump from $O(1)$ to arbitrarily small values
    upon perturbation of parameters, {\purp despite the ``nondegenerate'' nature of the curve of equilibria which enables a gradient inequality to hold with $O(1)$ constants.}
    Whether this phenomenon can occur more broadly in other kinds of gradient systems remains to be seen.

\section*{Acknowledgements}
  This work has been partially supported by the National Science Foundation under
  grants DMS 2106534, DMS 1814991 and DMS 2206069. The authors are grateful to Sir John Ball
  for remarks which helped clarify and correct several arguments. \blue The authors would also like to thank the anonymous referees for their helpful suggestions.\nc

\bibliographystyle{siam}
\bibliography{main_bib}

 \end{document}